\documentclass[3p,10pt,times]{elsarticle}

\usepackage{amssymb,amsmath,makeidx}
\usepackage[T1]{fontenc}
\usepackage{latexsym}
\usepackage{exscale}
\usepackage{xr}

\makeindex

\newcommand{\al}{\alpha}
\newcommand{\alvec}{{\vec{\al}}} 
\newcommand{\alvecbe}{{\vec{\al}^\frown\be}} 
\newcommand{\alvecga}{{\vec{\al}^\frown\ga}} 
\newcommand{\alvecpr}{{\vec{\al}'}}
\newcommand{\alvecpl}{{\vec{\al}^+}}

\newcommand{\alpr}{{\alpha^\prime}}

\newcommand{\ale}{{\al_1}}

\newcommand{\ali}{{\al_i}}

\newcommand{\alimin}{{\al_{i-1}}}

\newcommand{\aln}{{\al_n}}
\newcommand{\alne}{{\al_{n+1}}}

\newcommand{\be}{\beta}
\newcommand{\beplus}{{\be^+}}

\newcommand{\bevec}{{\vec{\be}}}
\newcommand{\bevecpl}{{\vec{\be}^+}}
\newcommand{\bevecpr}{{\vec{\be}'}}

\newcommand{\bebar}{\bar{\be}}

\newcommand{\bepr}{{\beta^\prime}}

\newcommand{\ga}{\gamma}
\newcommand{\gahat}{{\widehat{\ga}}}
\newcommand{\gapr}{{\ga^\prime}}

\newcommand{\gabar}{\bar{\ga}}
\newcommand{\gavec}{{\vec{\ga}}}
\newcommand{\gavecpr}{{\vec{\ga}'}}

\newcommand{\de}{\delta}

\newcommand{\devec}{{\vec{\de}}}

\newcommand{\devecpr}{{\vec{\de}^\prime}}
\newcommand{\De}{\Delta}

\newcommand{\eps}{\varepsilon}

\newcommand{\epsn}{\varepsilon_0}

\newcommand{\la}{\lambda}

\newcommand{\ka}{\kappa}
\newcommand{\om}{\omega}

\newcommand{\Om}{\Omega}

\newcommand{\nupr}{{\nu^\prime}}

\newcommand{\sivec}{{\vec{\si}}}

\newcommand{\sipr}{\si^\prime}

\newcommand{\tauti}{\tilde{\tau}}
\newcommand{\tauvec}{{\vec{\tau}}}

\newcommand{\taupr}{{\tau^\prime}}
\newcommand{\taubar}{\bar{\tau}}
\newcommand{\taubarpr}{\bar{\tau}^\prime}

\newcommand{\tautipr}{{\tauti^\prime}}

\newcommand{\taunpr}{{\tau^\prime_n}}

\newcommand{\taui}{{\tau_i}}
\newcommand{\tauie}{{\tau_{i+1}}}

\newcommand{\taunstar}{{\tau_n^\star}}
\newcommand{\taurstar}{{\tau_r^\star}}

\newcommand{\taunestar}{{\tau_{n+1}^\star}}
\newcommand{\si}{\sigma}
\newcommand{\tht}{\vartheta}

\newcommand{\xipr}{{\xi^\prime}}
\newcommand{\xivec}{{\vec{\xi}}}
\newcommand{\xivecpr}{{\vec{\xi}'}}
\newcommand{\ze}{\zeta}

\newcommand{\zevec}{{\vec{\ze}}}

\newcommand{\etapr}{{\eta^\prime}}

\newcommand{\etavec}{{\vec{\eta}}}

\renewcommand{\phi}{\varphi}

\newcommand{\N}{{\mathbb N}}
\newcommand{\Hz}{{\mathbb P}}

\newcommand{\Ez}{{\mathbb E}}
\newcommand{\Ezone}{{\mathbb E}_1}
\newcommand{\On}{{\mathrm{Ord}}}

\newcommand{\ANF}{{\mathrm{\scriptscriptstyle{ANF}}}}

\newcommand{\NF}{{\mathrm{\scriptscriptstyle{NF}}}}

\newcommand{\Lim}{\mathrm{Lim}}
\newcommand{\Image}{\mathrm{Im}}

\newcommand{\finsub}{\subseteq_\mathrm{fin}}
\newcommand{\logend}{{\mathrm{logend}}}
\newcommand{\sumend}{{\mathrm{end}}}

\newcommand{\leo}{\le_1}

\newcommand{\lo}{<_1}

\newcommand{\klex}{<_\mathrm{\scriptscriptstyle{lex}}}
\newcommand{\kglex}{\le_\mathrm{\scriptscriptstyle{lex}}}
\newcommand{\lepw}{\le_\mathrm{\scriptscriptstyle{pw}}}

\newcommand{\thtnod}{\tht_0}
\newcommand{\thte}{\tht_1}
\newcommand{\thti}{\tht_i}

\newcommand{\thtt}{\tht^\tau}
\newcommand{\thtti}{\tht^\taui}
\newcommand{\thttj}{\tht^{\tau_j}}
\newcommand{\thtal}{\tht^\al}

\newcommand{\T}{{\operatorname{T}}}
\newcommand{\Targ}[1]{{\operatorname{T}^{#1}}}

\newcommand{\Tt}{{\operatorname{T}^\tau}}
\newcommand{\Ttn}{{\operatorname{T}^{\tau_n}}}
\newcommand{\Tti}{{\operatorname{T}^{\taui}}}
\newcommand{\Ttvec}{{\operatorname{T}^\tauvec}}
\newcommand{\Tsivec}{{\operatorname{T}^\sivec}}

\newcommand{\ltvec}{{\operatorname{l}^\tauvec}}
\newcommand{\lsivec}{{\operatorname{l}^\sivec}}
\newcommand{\ltvecbe}{{\operatorname{l}^{\tauvec^\frown\be}}}

\newcommand{\Tal}{{\operatorname{T}^\al}}

\newcommand{\htarg}[1]{\operatorname{ht}_{#1}}
\newcommand{\hte}{{\operatorname{ht}_1}}

\newcommand{\ttau}{{\operatorname{t}^\al_\tau}}

\newcommand{\lh}{{\operatorname{lh}}}
\newcommand{\gs}{{\operatorname{gs}}}

\newcommand{\kvalbe}{\kappa^\alvec_\be}

\newcommand{\laga}{{\la_\ga}}

\newcommand{\laaln}{{\la_{\aln}}}

\newcommand{\lat}{{\la_\tau}}

\newcommand{\iotal}{\iota_{\tau,\al}}
\newcommand{\iotial}{\iota_{\tau_i,\al}}

\newcommand{\zetal}{{\ze^\tau_\al}}
\newcommand{\zetbe}{{\ze^\tau_\be}}

\newcommand{\lasi}{{\la^\si}}
\newcommand{\latal}{{\la^\tau_\al}}

\newcommand{\latbe}{{\la^\tau_\be}}

\newcommand{\piarg}[1]{{\pi_{#1}}}
\newcommand{\piarginv}[1]{{\pi^{-1}_{#1}}}
\newcommand{\pist}{\pi_{\si,\tau}}
\newcommand{\pivec}{\vec{\pi}}
\newcommand{\pivecind}{\vec{\pi}\mbox{-idx}}
\newcommand{\kavec}{\vec{\ka}}
\newcommand{\kavecind}{\vec{\ka}\mbox{-idx}}
\newcommand{\bfpiarg}[1]{{\vec{\pi}_{#1}}}
\newcommand{\bfkaarg}[1]{{\vec{\ka}_{#1}}}
\newcommand{\pistinv}{\pi^{-1}_{\si,\tau}}

\newcommand{\Part}{\operatorname{Par}^\tau}

\newcommand{\kpom}{{\operatorname{KP}\!\om}}
\newcommand{\kplnod}{{\operatorname{KP}\!\ell_0}}

\newcommand{\Rtwo}{{{\cal R}_2}}
\newcommand{\Ctwo}{{{\cal C}_2}}

\newcommand{\Ronepl}{{{\cal R}_1^+}}
\newcommand{\Rtwopl}{{{\cal R}_2^+}}

\newcommand{\Core}{{\operatorname{Core}}}
\newcommand{\id}{\operatorname{id}}

\newcommand{\bardot}{\bar{\cdot}}

\newcommand{\aeq}{\Leftrightarrow}
\newcommand{\andsp}{\:\&\:}

\newcommand{\sub}{\subseteq}
\newcommand{\propersub}{\subsetneq}

\newlength{\hilflh}
\newcommand{\hilfminus}[1]{
  \settowidth{\hilflh}{$#1-$}\mbox{$#1-\hspace{-0.5\hilflh}
  \makebox[0pt]{\raisebox{0.24\hilflh}{$#1\cdot$}}\hspace{0.5\hilflh}$}}
\newcommand{\minusp}{\mathbin{\mathchoice {\hilfminus{\displaystyle}}
  {\hilfminus{\textstyle}}{\hilfminus{\scriptstyle}}
  {\hilfminus{\scriptscriptstyle}}}}

\newtheorem{theo}{Theorem}[section]
\newtheorem{claim}[theo]{Claim}
\newtheorem{cor}[theo]{Corollary}
\newtheorem{lem}[theo]{Lemma}
\newproof{proof}{\bf Proof}
\newproof{firstclaimproof}{\bf Proof of Claim \ref{mainlinecovclaim}}
\newproof{secondclaimproof}{\bf Proof of Claim \ref{extensionpartclaim}}
\newdefinition{defi}[theo]{Definition}
\newdefinition{prop}[theo]{Proposition}
\newdefinition{rmk}[theo]{Remark}

\newcommand{\oneinf}{1^\infty}

\newcommand{\tauinf}{\tau^\infty}
\newcommand{\alinf}{\al^\infty}

\newcommand{\chis}{\chi^\si}
\newcommand{\chit}{\chi^\tau}

\newcommand{\chial}{\chi^\al}

\newcommand{\chiga}{\chi^\ga}
\newcommand{\mutal}{\mu^\tau_\al}

\newcommand{\mualn}{\mu_\aln}

\newcommand{\mutali}{\mu^\tau_\ali}

\newcommand{\muga}{\mu_\ga}

\newcommand{\rhoalbe}{{\varrho^\al_\be}}

\newcommand{\MNF}{{\mathrm{\scriptscriptstyle{MNF}}}}

\newcommand{\Mz}{{\mathbb M}}

\newcommand{\trs}{{\mathrm{ts}}}
\newcommand{\trst}{{\mathrm{ts}^\tau}}

\newcommand{\trsaln}{{\mathrm{ts}^\aln}}

\newcommand{\cs}{{\mathrm{cs}}}
\newcommand{\cspr}{{\mathrm{cs}^\prime}}

\newcommand{\tc}{{\mathrm{tc}}}
\newcommand{\TC}{{\mathrm{TC}}}

\newcommand{\RS}{\mathrm{RS}}

\newcommand{\letwo}{\le_2}
\newcommand{\ktwo}{<_2}

\newcommand{\lSeq}{\mathrm{lSeq}}

\newcommand{\dom}{\mathrm{dom}}

\newcommand{\dpf}{\mathrm{dp}}

\newcommand{\nuval}{\nu^\alvec}

\newcommand{\nuvalbe}{\nu^\alvec_\be}

\newcommand{\alcp}[1]{{\al_{#1}}}
\newcommand{\becp}[1]{{\be_{#1}}}
\newcommand{\gacp}[1]{{\ga_{#1}}}
\newcommand{\decp}[1]{{\de_{#1}}}
\newcommand{\sicp}[1]{{\si_{#1}}}

\newcommand{\taucp}[1]{{\tau_{#1}}}
\newcommand{\taucppr}[1]{{\tau^\prime_{#1}}}

\newcommand{\xicp}[1]{{\xi_{#1}}}

\newcommand{\tauticp}[1]{{\tauti_{#1}}}

\newcommand{\rhoargs}[2]{\varrho^{#1}_{#2}}

\newcommand{\ov}{\mathrm{o}}

\newcommand{\ordcp}[1]{{\mathrm{o}_{#1}}}

\newcommand{\ktc}{<_\mathrm{TC}}
\newcommand{\letc}{\le_\mathrm{TC}}

\newcommand{\cml}{\operatorname{cml}}
\newcommand{\gbo}{\operatorname{gbo}}
\newcommand{\predec}{\operatorname{pred}}
\newcommand{\predecs}{\operatorname{Pred}}

\newcommand{\TS}{\operatorname{TS}}

\newcommand{\mts}{\operatorname{mts}}
\newcommand{\mtsal}{\operatorname{mts}^\al}

\newcommand{\mtsga}{\operatorname{mts}^\ga}

\newcommand{\hop}{\operatorname{h}}

\newcommand{\hbe}{\operatorname{h}_\be}

\newcommand{\sk}{\operatorname{sk}}
\newcommand{\skbe}{\operatorname{sk}_\be}

\newcommand{\ec}{\operatorname{ec}}
\newcommand{\me}{\operatorname{me}}
\newcommand{\mepl}{\operatorname{me}^+}

\newcommand{\param}{\operatorname{par}}
\newcommand{\maxparam}{\operatorname{mp}}
\newcommand{\parind}{\operatorname{pi}}
\newcommand{\db}{\operatorname{db}}
\newcommand{\dc}{\operatorname{dc}}
\newcommand{\ds}{\operatorname{ds}}
\newcommand{\mq}{{\operatorname{mq}}}

\newcommand{\mqtvec}{{\operatorname{mq}^\tauvec}}

\newcommand{\Mcl}{M^{\operatorname{cl}}}
\newcommand{\Mecl}{M^{\operatorname{ecl}}}
\newcommand{\Mpr}{M^\prime}

\newcommand{\bp}{\operatorname{bp}}
\newcommand{\vc}{\operatorname{vc}}
\newcommand{\vcal}{\vc_\al}
\newcommand{\vcalpr}{\vc_{\alpr}}

\def\vec#1{\mathchoice{\mbox{\boldmath$\displaystyle#1$}}
{\mbox{\boldmath$\textstyle#1$}}
{\mbox{\boldmath$\scriptstyle#1$}}
{\mbox{\boldmath$\scriptscriptstyle#1$}}}

\newcommand{\bs}{{\mathrm{bs}}}
\newcommand{\bspr}{\mathrm{bs}^\prime}

\journal{Annals of Pure and Applied Logic}

\begin{document}

\begin{frontmatter}

\title{Pure patterns of order 2}

\author{Gunnar Wilken}

\ead{wilken@oist.jp}

\address{Structural Cellular Biology Unit\\
Okinawa Institute of Science and Technology\\
1919-1 Tancha, Onna-son, 904-0495 Okinawa, Japan}

\begin{abstract}
We provide mutual elementary recursive order isomorphisms between classical ordinal notations, based on Skolem hulling, and
notations from pure elementary patterns of resemblance of order $2$,  
showing that the latter characterize the proof-theoretic ordinal $\oneinf$\index{$1$@$\oneinf$} of the fragment $\Pi^1_1$-$\mathrm{CA}_0$ 
of second order number theory, or equivalently the set theory $\kplnod$. 
As a corollary, we prove that Carlson's result on the well-quasi orderedness of respecting forests of order $2$ 
implies transfinite induction up to the ordinal $\oneinf$.
We expect that our approach will facilitate analysis of more powerful systems of patterns.
\end{abstract}


\end{frontmatter}

\section{Introduction}
Elementary patterns of resemblance were discovered and then systematically introduced by Timothy J.\ Carlson, \cite{C00,C01,C09}, 
as an alternative approach to recursive systems of ordinal notations. Elementary patterns constitute the basic levels of Carlson's 
programmatic approach, {\it patterns of embeddings}, which is inspired by G\"odel's program of using large cardinals to solve mathematical incompleteness, see e.g.\ \cite{F96,K09}. It follows heuristics that axioms of infinity are
in close correspondence with ordinal notations. The long-term goal of patterns of embeddings is therefore to find an ultra-finestructure for large cardinal axioms based on embeddings, thereby ultimately complementing inner model theory.

Patterns of resemblance, which instead of involving codings of embeddings, rely upon binary relations coding the property of elementary 
substructure of increasing complexity, are first steps to investigate patterns.   
Inspired by the notion of elementary substructure along ordinals as set-theoretic objects, ordinal notations
in terms of elementary patterns intrinsically carry semantic content. However, Carlson made the intriguing observation that
patterns have simple, finitely combinatorial characterizations called \emph{respecting forests}.  

The present article focuses on elementary patterns of order $2$. Recalling from the introduction to \cite{W}, let 
$\Rtwo=\left(\On;\le,\leo,\letwo\right)$\index{$\r$@$\Rtwo$}\index{$\le_i$} be the structure of ordinals 
with standard linear ordering $\le$ and partial orderings $\le_1$ and $\le_2$, simultaneously defined by induction on $\be$ in  
\[\al\le_i\be:\aeq \left(\al;\le,\le_1,\le_2\right) \preceq_{\Sigma_i} \left(\be;\le,\le_1,\le_2\right)\]
where $\preceq_{\Sigma_i}$ is the usual notion of $\Sigma_i$-elementary substructure (without bounded quantification), see \cite{C99,C01}
for fundamentals and groundwork on elementary patterns of resemblance.
Pure patterns of order $2$\index{patterns, pure patterns} are the finite isomorphism types of $\Rtwo$. The \emph{core}\index{core}
of $\Rtwo$ consists of the union of \emph{isominimal realizations}\index{isominimal realization} of these patterns within $\Rtwo$, where a finite 
substructure of $\Rtwo$ is called isominimal, if it is pointwise minimal (with respect to increasing enumerations) 
among all substructures of $\Rtwo$ isomorphic to it, and where an isominimal substructure of $\Rtwo$ realizes a pattern $P$, if
it is isomorphic to $P$. It is a basic observation, cf.\ \cite{C01}, that the class of pure patterns of order $2$ is 
contained in the class $\mathcal{RF}_2$ of \emph{respecting forests of order $2$}:\index{respecting forest}
finite structures $P$ over the language $(\le_0,\leo,\letwo)$ where $\le_0$ is a linear ordering and $\leo,\letwo$ are forests such that 
$\letwo\subseteq\leo\subseteq\le_0$ and $\le_{i+1}$ \emph{respects} $\le_i$, i.e.\ $p\le_i q\le_i r\andsp p\le_{i+1}r$ implies $p\le_{i+1}q$
for all $p,q,r\in P$, for $i=0,1$.  

In \cite{CWc} we showed that every pattern has a cover below $1^\infty$, the least such ordinal.
As outlined in \cite{W}, an order isomorphism (embedding) is a cover (covering, respectively)\index{covering,cover} 
if it maintains the relations $\le_1$ and $\le_2$.
The ordinal of $\kplnod$, which axiomatizes limits of models of Kripke-Platek set theory with infinity, 
is therefore least such that there exist arbitrarily long finite $\le_2$-chains.
Moreover, by determination of enumeration functions of (relativized) connectivity components of $\le_1$ and $\le_2$, we
were able to describe these relations in terms of classical ordinal notations. The central observation in connection with
this is that every ordinal below $1^\infty$ is the greatest element in a $\le_1$-chain in which $\le_1$- and $\le_2$-chains alternate,
thus providing a formalism that allows precise localization of ordinals in terms of relativized connectivity components of the
relations $\leo$ and $\letwo$.
We called such chains \emph{tracking chains}, as they provide all $\le_2$-predecessors and the greatest $\le_1$-predecessors
insofar as they exist. 

In \cite{W} we showed that the arithmetical characterization of the structure $\Rtwo$ up to the ordinal $\oneinf$, which we denoted
as $\Ctwo$\index{$\Ctwo$}, is an elementary recursive structure. This guarantees the elementary recursiveness of the order isomorphisms between
hull and pattern notations given here.

From these preparations we devise here an algorithm that assigns an isominimal realization within $\Ctwo$ to each respecting forest
of order $2$, thereby showing that each such respecting forest is in fact (up to isomorphism)
a pure pattern of order $2$. It turns out that isominimal realizations are pointwise minimal among all covers of the given forest.
We therefore derive a method that calculates ordinals coded in pattern notations in terms of familiar hull notations, see \cite{W07a}.

The notion of closure introduced here further allows us to provide pattern notations for finite sets of ordinals below $\oneinf$.
We are going to define an elementary recursive function that assigns describing patterns $P(\al)$ to
ordinals $\al\in 1^\infty$. Recalling again from \cite{W}, a descriptive pattern\index{patterns, pure patterns!descriptive patterns} 
for an ordinal $\al$ is a pattern, the isominimal realization of which contains $\al$. 
Descriptive patterns are given in a way that makes a canonical choice for normal
forms, since in contrast to the situation in $\Ronepl$, cf.\ \cite{W07c,CWa}, there is no unique notion of normal form in $\Rtwo$.
The chosen normal forms are of least possible cardinality.

The mutual order isomorphisms between hull and pattern notations in the present article enable classification of a new independence
result for $\kplnod$, as was already announced \cite{W}. 
We demonstrate that Carlson's result in \cite{C16}, according to which the collection of respecting forests of order $2$ is well-quasi-ordered with 
respect to coverings, cannot be proven in $\kplnod$ or, equivalently, in the restriction $\Pi^1_1\mathrm{-CA}_0$ of second order number theory to 
$\Pi^1_1$-comprehension and set induction. On the other hand,
we know that transfinite induction up to the ordinal $1^\infty$ of $\kplnod$ suffices to show that every pattern is covered \cite{CWc}.

\section{Preliminaries}\label{prelimsec}
For a general introduction to proof theory and ordinal notation systems, see Pohlers \cite{P09}. Classical notations 
based on Skolem hulling \cite{P09} that are used here (relativized notation systems $\Tt$\index{$\Tt$}, 
collapsing functions $\thtt$,$\thti$\index{$\thtt$,$\thti$})  
were provided in \cite{W07a} together with structural insights particularly useful
in analysis of patterns of resemblance, first demonstrated in \cite{W07b}.
\cite{W07a} introduces frequently used ordinal measures, e.g.\ $\htarg{}$\index{$\htarg{}$}, 
transformations, e.g.\ $\iotal$\index{$\iotal$}, $\pist$\index{$\pist$}, 
and arithmetical operations $\bardot$\index{$\bardot$}, $\zetal$\index{$\zetal$}, $\latal$\index{$\latal$}.
A summary of this toolkit can be found in \cite{W07c}, where the core of the structure
$\Ronepl$\index{$\r$@$\Ronepl$} was analyzed. This was further enhanced in Sections 5 and 6 of \cite{CWa}. 

This article builds upon the results, arithmetical tools, and terminology of \cite{CWc} and \cite{W}.
For ordinal arithmetical functions and operators specific to the analysis of patterns of order $2$ such as $\chit$\index{$\chit$}, 
$\rhoargs{\tau}{\al}$\index{$\rhoargs{\tau}{\al}$}, $\mutal$\index{$\mutal$}, $\widehat{\cdot}$\index{$\widehat{\cdot}$}
see Section 3 of \cite{CWc}.
The central notion is that of \emph{tracking chains},\index{tracking chains, $\tc$, $\TC$} introduced in Definitions 5.1, 5.2, and 6.1 of \cite{CWc}, 
and thoroughly explained and analyzed in Section 5 of \cite{W}. It provides a detailed description of the relations $\leo$ and $\letwo$ 
in terms of (relativized) connectivity components, thereby providing ``addresses'' for the ordinals below $\oneinf$ in terms of 
nested components of $\le_i$, $i=1,2$.
Corollary 5.8 of \cite{W}, here \ref{elemreccharcor}, summarizes the arithmetical, and even syntactic, 
characterization of the semantic relations $\le_i$, coding $\Sigma_i$-elementarity within $\Rtwo$, up to $\oneinf$.
Notions of closedness and closure introduced in the present article build upon the notion of (relativized) spanning sets of 
tracking chains, introduced in Definitions 5.1, 5.2, and 5.3 of \cite{W}.

For the reader's convenience we give a review of the central notions and results from \cite{W} and provide an index. Notions and abbreviations
from \cite{CWc} that are not explained here can be quickly accessed through the index of \cite{CWc}, and similarly for more basic 
notions from \cite{W07a}.

\subsection{Tracking sequences and connectivity components}\label{tsccsec}
To begin, recall the notion of localization\index{localization} from Section 4 of \cite{W07a} (Definition 4.6). 
Accompanying the notion of \emph{tracking sequence}\index{tracking sequence, $\trs$, $\TS$}, see Definitions 3.13 and 4.2
\footnote{Note that in Definition 4.2 of \cite{CWc} the restriction $\aln\ge\tau$ was not meant to be applied in case of $n>1$.} 
of \cite{CWc}, we have
\begin{defi}[corrected 3.6 of \cite{W}]\label{mts}\index{$\mts$}
Let $\tau\in\Ezone$, $\al\in\Ez\cap(\tau,\tauinf)$, $\be\in\Hz\cap[\al,\alinf)$, and let $\al=\al_0,\ldots,\al_{n}=\be$ be the 
$\al$-localization of $\be$.
If there exists the least index $i$ such that $0\le i<n$ and $\ali<\be\le\mutali$, then 
\[\mtsal(\be):=\mtsal(\ali)^\frown(\be),\]
otherwise $\mtsal(\be):=(\al)$.
\end{defi}
This notion has been discussed in Subsection 3.1 of \cite{W} and, together with Definitions 4.3\index{reference sequence, $\RS$} 
and 4.9 of \cite{CWc},\index{$\sk$} gives rise to the following
\begin{defi}[3.11 of \cite{W}]\label{hgamalbe}\index{$\hbe$}
Let $\alvec^\frown\ga\in\RS$ and $\be\in\Mz$.
\begin{enumerate}
\item If $\be\in(\ga,\gahat)$, let $\mtsga(\be)=\etavec^\frown(\eps,\be)$ and define
$\hbe(\alvecga):=\alvec^\frown\etavec^\frown\skbe(\eps)$.
\item If $\be\in(1,\ga]$ and $\be\le\muga$ then 
$\hbe(\alvecga):=\alvec^\frown\skbe(\ga)$.
\item If $\be\in(1,\ga]$ and $\be>\muga$ then 
$\hbe(\alvecga):=\alvec^\frown\ga$.
\end{enumerate}
\end{defi}
With this preparation at hand, the crucial definition of Section 3 of \cite{W} reads 
\begin{defi}[3.14 of \cite{W}]\label{odef}
Let $\alvecbe\in\TS$, where $\alvec=(\ale,\ldots,\aln)$, $n\ge 0$, $\be=_\MNF\be_1\cdot\ldots\cdot\be_k$,
and set $\al_0:=1$, $\alne:=\be$, $h:=\hte(\ale)+1$, and
$\gavec_i:=\trs^{\alimin}(\ali)$, $i=1,\ldots,n$,
\[\gavec_{n+1}:=\left\{\begin{array}{ll}
            (\be)&\mbox{if }\be\le\aln\\[2mm]
            \trsaln(\be_1)^\frown\be_2&\mbox{if } k>1,\be_1\in\Ez^{>\aln}\andsp\be_2\le\mu_{\be_1}\\[2mm]
            \trsaln(\be_1)&\mbox{otherwise,}
            \end{array}\right.\]
and write $\gavec_i=(\ga_{i,1},\ldots,\ga_{i,m_i})$, $i=1,\ldots,n+1$.
We then define the set\index{$\lSeq$}
\[\lSeq(\alvecbe):=(m_1,\ldots,m_{n+1})\in[h]^{\le h}\]
of sequences of natural numbers $\le h$ of length at most $h$, ordered lexicographically. 
Let $\bepr:=1$ if $k=1$ and $\bepr:=\be_2\cdot\ldots\cdot\be_k$ otherwise. 
We define $\ov(\alvecbe)$\index{$\ov$} recursively in $\lSeq(\alvecbe)$, as well as auxiliary parameters 
$n_0(\alvecbe)$ and $\ga(\alvecbe)$, which are set to $0$ where not defined explicitly.
\begin{enumerate}
\item $\ov((1)):=1$.
\item If $n>1$ and $\be_1\le\aln$, then $\ov(\alvecbe):=_\NF\ov(\alvec)\cdot\be$.
\item If $\be_1\in\Ez^{>\aln}$, $k>1$, and $\be_2\le\mu_{\be_1}$, then
set $n_0(\alvecbe):=n+1$, $\ga(\alvecbe):=\be_1$, and define
\[\ov(\alvecbe):=_\NF\ov(\hop_{\be_2}(\alvec^\frown\be_1))\cdot\bepr.\] 
\item Otherwise. Then setting
\[n_0:=n_0(\alvecbe):=\max\left(\{i\in\{1,\ldots,n+1\}\mid m_i>1\}\cup\{0\}\right),\]
define 
\[\ov(\alvecbe):=_\NF\left\{\begin{array}{ll}
            \be&\mbox{if } n_0=0\\[2mm]
            \ov(\hop_{\be_1}({\alvec_{\restriction_{n_0-1}}}^\frown\ga))\cdot\be&\mbox{if } n_0>0,
            \end{array}\right.\] 
where $\ga:=\ga(\alvecbe):=\ga_{n_0,m_{n_0}-1}$.
\end{enumerate}
\end{defi}

\begin{rmk}[\cite{W}]  As indicated in writing $=_\NF$ in the above definition, we obtain terms in multiplicative normal form
denoting the values of $\ov$. The {\bf fixed points of $\ov$}, i.e.\ those $\alvecbe$ that satisfy $\ov(\alvecbe)=\be$ are therefore
characterized by 1.\ and 4.\ for $n_0=0$.
\end{rmk}

Once having noticed that the proof of Lemma 3.15 of \cite{CWc} actually proceeds by induction along the inductive definition of $\trst$, hence
along term decomposition, we can now, in an elementary recursive fashion, characterize the enumeration functions of relativized connectivity 
components introduced in Section 4 of \cite{CWc} (Definition 4.4), as carried out in detail in Section 4 of \cite{W}.  
\begin{defi}[4.1 of \cite{W}]\label{kappanuprincipals}
Let $\alvec\in\RS$ where $\alvec=(\ale,\ldots,\aln)$, $n\ge 0$, $\al_0:=1$.
We define $\kvalbe$\index{$\kvalbe$} and $\nuvalbe$\index{$\nuvalbe$} for certain additive principal $\be$ as follows, writing $\kappa_\be$ 
instead of $\kappa^{()}_\be$.\\[2mm]
{\bf Case 1:} $n=0$. For $\be<\oneinf$ define \[\kappa_\be:=\ov((\be)).\]
{\bf Case 2:} $n>0$. For $\be\le \mualn$, i.e.\ $\alvec^\frown\be\in\TS$, define \[\nu^\alvec_\be:=\ov(\alvecbe).\]
$\kvalbe$ for $\be\le \laaln$ is defined by cases. If $\be\le\aln$ let $i\in\{0,\ldots,n-1\}$ be maximal such that $\ali<\be$.
If $\be>\aln$ let $\be=_\MNF\be_1\cdot\ldots\cdot\be_k$ and set $\bepr:=(1/\be_1)\cdot\be$. 
\[\kvalbe:=\left\{\begin{array}{ll}
            \ka^{{\alvec_{\restriction_i}}}_\be&\mbox{if } \be\le\aln\\[2mm]
            \ov(\alvec)\cdot\bepr&\mbox{if } \be_1=\aln\andsp k>1\\[2mm]
            \ov(\alvecbe)&\mbox{if } \be_1>\aln.
            \end{array}\right.\]
\end{defi}
Lemma 4.10 of \cite{CWc} shows that $\trs$ and $\ov$ invert each other, which is shown in 
Theorems 3.19 and 3.20 of \cite{W} by induction along term decomposition.
Note that in the case where $\be_1=\aln$ and $k>1$ we have $\kvalbe=\ov(\alvec^\frown\bepr)=\nuval_\bepr$.
We thus obtain representations of the functions $\ka$ and $\nu$ in multiplicative normal form.

The conservative extension of $\kappa$ and $\nu$ to their entire domain as well as
the definition of $\dpf$\index{$\dpf$}, which are in accordance with Definition 4.4 of \cite{CWc}, can now be carried out as in 
\cite{W} (Definitions 4.4 and 4.5) by recursion on the following simple term measure.
\begin{defi}[4.3 of \cite{W}]\label{Ttauvec}\index{$\Ttvec$}
Let $\tauvec\in\RS$, $\tauvec=(\tau_1,\ldots,\tau_n)$, $n\ge0$, $\tau_0:=1$.
The term system $\Ttvec$ is obtained from $\Ttn$ by successive substitution of parameters in $(\taui,\tauie)$ by their $\Tti$-representations,
for $i=n-1,\ldots,0$. The parameters $\taui$ are represented by the terms $\thtti(0)$.
The length $\ltvec(\al)$\index{$\ltvec$} of a $\Ttvec$-term $\al$ is defined inductively by 
\begin{enumerate}
\item $\ltvec(0):=0$,
\item $\ltvec(\be):=\ltvec(\ga)+\ltvec(\de)$ if $\be=_\NF\ga+\de$, and 
\item $\ltvec(\tht(\eta)):=\left\{\begin{array}{l@{\quad}l}
1&\mbox{ if }\quad\eta=0\\
\ltvec(\eta)+4&\mbox{ if }\quad\eta>0
\end{array}\right.$\\[2mm] 
where $\tht\in\{\tht^{\tau_i}\mid 0\le i\le n\}\cup\{\tht_{i+1}\mid i\in\N\}$.
\end{enumerate}
\end{defi}
Recall that for $\be=\thtt(\De+\eta)$, where $\tau\in\Ezone=\Ez\cup\{1\}$ and, according to Convention 4.1 of \cite{W07a},
$\eta<\Om_1$ and $\Om_1\mid\De$, i.e.\ $\De$ is a (possibly zero) multiple of $\Om_1$, we have\index{$\zetal$}
\begin{equation}\label{zetaldefi}
\zetbe:=\left\{\begin{array}{ll}
                    \logend(\eta) & \mbox{ if } \eta<\sup_{\si<\eta}\thtt(\De+\si)\\
                    0 & \mbox{ otherwise.}
                 \end{array}\right.
\end{equation}
The ordinal function $\logend$\index{$\logend$}, which picks the exponent of the last additive component, is characterized by
\begin{equation}\label{logendchar}
  \logend(\al+\be)=\left\{\begin{array}{ll}
                      \be & \mbox{ if }\De>0\\
                      \eta+1 & \mbox{ if } \De=0 \mbox{ and } \eta=\eps+k\mbox{ where }\eps\in\Ez^{>\tau} \mbox{ and } k<\om\\
                      (-1+\tau)+\eta & \mbox{ otherwise.}
                      \end{array}\right.
\end{equation}
In the case $\De = 0$ we have\index{$\log$}
\begin{equation}\label{logred}
\log((1/\tau)\cdot\be)=\left\{\begin{array}{l@{\quad}l}
\eta+1&\mbox{ if }\eta=\eps+k\mbox{ where }\eps\in\Ez^{>\tau} \mbox{ and } k<\om\\[2mm]
\eta&\mbox{ otherwise.}
\end{array}\right.
\end{equation}
We now recall the $\bardot$ operator, see Section 8 of \cite{W07a}, as given by Definition 5.1 of \cite{CWa}. 
If $\be>\tau$, let $\tau=\be_0,\ldots,\be_n=\be$
be its $\tau$-localization. If $\eta>0$, write $\eta=_\NF\etapr+\eta_0$ where $\eta_0\in\Hz$ and either $\etapr=0$ or $\etapr\ge\eta_0$,
otherwise set $\etapr$ and $\eta_0$ to $0$. 
$\bebar\in[\tau,\be)$ is then defined as follows.\index{$\bardot$}
\begin{equation}\label{barop}
  \bebar:= \left\{\begin{array}{ll}
                     \thtt(\De+\etapr)  & \mbox{ if }\eta_0=1\mbox{ or }\etapr<\sup_{\si<\etapr}\thtt(\De+\si)\\
                     \be_{n-1} & \mbox{ otherwise.}
                  \end{array}\right.
\end{equation}
With operators $\iota$ and $\la$ as in Definitions 7.1 and 7.5 of \cite{W07a} we have the following estimations of term complexity
as stated in a remark following Definition 4.3 in \cite{W}.
\begin{enumerate}
\item For $\be=\tht^{\tau_n}(\De+\eta)\in\Ez$ such that $\be\le\mu_{\tau_n}$ we have
\begin{equation}\label{iotalen} 
\ltvec(\De)=\ltvecbe(\iota_{\tau_n,\be}(\De))<\ltvec(\be).
\end{equation} 
\item For $\be\in\Ttvec\cap\Hz^{>1}\cap\Om_1$ let $\tau\in\{\tau_0,\ldots,\tau_n\}$ be maximal such that $\tau<\be$. 
Clearly,
\begin{equation}\label{barlen} \ltvec(\bebar) < \ltvec(\be),
\end{equation}
and
\begin{equation}\label{zelen}
\ltvec(\zetbe) < \ltvec(\be),
\end{equation}
In case of $\be\not\in\Ez$ we have
\begin{equation}\label{loglen} \ltvec(\log(\be)), \ltvec(\log((1/\tau)\cdot\be)) < \ltvec(\be),
\end{equation}
and for $\be\in\Ez$ we have
\begin{equation}\label{lalen} 
\ltvecbe(\latbe) < \ltvec(\be).
\end{equation}
\end{enumerate}

All results from Section 4 of \cite{CWc} have been re-established using induction on term complexity ($\ltvec$) in Section 4 of \cite{W}.
We now see that the functions $\ka$, $\nu$, and $\dpf$ completely resolve into summations of $\ov$-terms, which in turn are given in
multiplicative normal form and which increase with respect to the lexicographic ordering on tracking sequences. We therefore obtain
an explicit additive normal form representation for these enumeration functions.

\subsection{Tracking chains}\label{tcacsec}
Recall Definitions 5.1, 5.2, and 5.3 of \cite{CWc} for the notions of tracking chain, $\cml$\index{$\cml$}, (maximal) extension ($\me$)\index{$\me$},
and characteristic sequence\index{characteristic sequence, $\cs$} ($\cs$). 
While a proof of termination of 
arbitrary (non-maximal) stepwise extensions of tracking chains requires strong induction, as used in the proof of part a) of Lemma 5.4 
in \cite{CWc}, the termination of stepwise maximal extensions, as in Definition 5.2 of \cite{CWc}, is easily seen when applying the measure
$\ltvec$ (Definition \ref{Ttauvec}) from the second step on, as clause 2.3.1 of Definition 5.2 of \cite{CWc} can only be applied at the 
beginning, as mentioned in \cite{W}. In this context we will give an alternative, more instructive, proof of Lemma 5.5 b) and Corollary 5.6
of \cite{CWc} shortly. 

Notice that in \cite{CWc}, Lemmas 5.7, 5.8, and 5.10 do not depend on Lemmas 5.4, 5.5, and Corollary 5.6.   
Part a) of Lemma 5.7 depends on part b) of Lemma 3.12 of \cite{CWc}, the second part of which misses the condition $\chial(\be)=0$ as mentioned
in Section 2 of \cite{W}.
It should read as follows: For $\be<\mutal$ such that $\chial(\be)=0$ we even have $\rhoalbe+\al\le\latal$. This missing condition, however,
is fulfilled in the proof of Lemma 5.7 (cf.\ the definition of the delimiters $\rho_i$\index{$\rho_i$} in Definition 5.1 of \cite{CWc}).

As outlined in Subsection 5.1 of \cite{W}, the proof of Lemma 5.12 of \cite{CWc} actually proceeds by induction on the number of 
$1$-step extensions rather than by induction on $\cspr(\alvec)$ along $\klex$. Furthermore, the proof of Lemma 6.2 of \cite{CWc} proceeds 
by induction on the length of the additive decomposition of $\al$. The evaluation function $\ov$ and its shorthand 
$\tilde{\cdot}$\index{$\tilde{\cdot}$} as in 
Definition 5.9 of \cite{CWc}, which are applied in Definition 6.1 of \cite{CWc}, have been seen to be elementary recursive here, 
with detailed proofs in \cite{W}.

Corollary 6.6 of \cite{CWc} characterizes $\tc(\al+\dpf(\alvec))$ for $\al<\oneinf$ with tracking chain $\alvec:=\tc(\al)$.
Contrary to the formulation in \cite{CWc} this is equal to $\alvec_{\restriction_{i,j+1}}[\alcp{i,j+1}+1]$ if
either $(i,j):=\cml(\alvec)$ exists or $m_n>1$, $(i,j+1)=(n,m_n)$ $\andsp$ $\taucp{i,j+1}<\mu_\taucp{i,j}$. Otherwise it is equal to $\me(\alvec)$,
as stated. The same case distinction applies to the statement regarding $\be$. The adjustments in the proofs of Corollary 6.6 and 7.13 of 
\cite{CWc} are straightforward.

The following definition is a preparation for a more explicit proof of Lemma 5.5 b) and Corollary 5.6 of \cite{CWc}, see Case 1 below,
and for \emph{base minimization} introduced in Definition \ref{pitcdefi} in Section \ref{spcltcsec}, see Case 2. Our aim is to
keep track of the indices of $\leo$-components along $\leo$-chains and relate them to term decomposition.
\begin{defi}\label{sivecdefi}
Let $\alvec\in\TC$ with components $\alvec_i=(\alcp{i,1},\ldots,\alcp{i,m_i})$ for $i=1,\ldots,n$.
\\[2mm]
{\bf Case 1:} $(i,j):=\cml(\alvec)$ exists. Set $\tau:=\taucp{i,j+1}$, $\taupr:=\taucp{i,j}$, and $\alvecpr:=\alvec_{\restriction_{i,j+1}}$. 
Let $\tauvec$ be the chain associated with $\alvecpl:=\mepl(\alvecpr)=\me(\alvec)$. 
Define $\sivec=(\si_1,\ldots,\si_k):=\cs(\alvec_{\restriction_{i,j}})$, which by
Lemma 5.10 of \cite{CWc} is equal to $\trs(\tauticp{i,j})$.
\[\si_{k+1}:=\left\{\begin{array}{ll}
                      \tau & \mbox{ if }\tau\in\Ez^{>\taupr}\\[1mm]
                      \sumend(\rhoargs{\taupr}{\tau}) & \mbox{ otherwise,}
                    \end{array}\right.\]
which corresponds to $\taucp{i,j+1}$ in the former and to $\taucp{i+1,1}$ in the latter case.                   
\\[2mm]
{\bf Case 2:} $\cml(\alvec)$ does not exist. Set $\tau:=\taucp{n,m_n}$ and $\taupr:=\tau_{(n,m_n)^\prime}$. 
Define $\alvecpl:=\me(\alvecpr)$ where
\[\alvecpr:=\left\{\begin{array}{ll}
              \alvec & \mbox{if }m_n=1\mbox{ or }\tau=1 \\[1mm]
              {\alvec_{\restriction_{n-1}}}^\frown(\alcp{n,1},\ldots,\alcp{n,m_n},\mu_\tau) &
                 \mbox{if }m_n>1\mbox{ and }\tau\in\Ez^{>\taupr}\\[1mm] 
              \alvec^\frown(\rhoargs{\taupr}{\tau}) & \mbox{otherwise.} 
            \end{array}\right.\]
Extend $\tauvec$ to be the chain associated with $\alvecpl$. Define $\sivec=(\si_1,\ldots,\si_k):=\cs(\alvec_{\restriction_{(n,m_n)^\prime}})$, which by Lemma 5.10 of \cite{CWc} is equal to $\trs(\tautipr)$, unless $(n,m_n)^\prime=(1,0)$. Let 
\[\si_{k+1}:=\left\{\begin{array}{ll}
                      \tau & \mbox{ if }m_n=1\mbox{ or }\tau=1\mbox{ or }\tau\in\Ez^{>\taupr}\\[1mm]
                      \sumend(\rhoargs{\taupr}{\tau}) & \mbox{ otherwise,}
                    \end{array}\right.\]
which corresponds to $\taucp{n,m_n}$ in the former and to $\taucp{n+1,1}$ in the latter case.                   
\\[2mm]
Now, for either case, define $\bs_h:=\sivec_{\restriction_h}=(\si_1,\ldots,\si_h)$ for $h=0,\ldots,k$, and 
$\bspr_h:=\sivec_{\restriction_{h-1}}$ for $h=1,\ldots,k$. If $\si_{l+1}$ has been defined for some $l\ge k$, 
then write $\bs_l=(\rho_1,\ldots,\rho_r)$, set $\rho_0:=1$, 
and let $h\in\{0,\ldots,r\}$ be maximal such that $\rho_h\le\si_{l+1}$. Define $\sipr_{l+1}:=\rho_h$ and $\bspr_{l+1}:=(\rho_1,\ldots,\rho_h)$.
\\[2mm]
Let $l\ge k+1$ be such that $\si_l,\sipr_l,\bs_l,\bspr_l$\index{$\bs$} are defined. If $\si_l=\sipr_l$, the sequence $\sivec=(\si_1,\ldots,\si_l)$ is 
complete. Otherwise we have $\sipr_l<\si_l$, hence $\sumend(\log((1/\sipr_l)\cdot\si_l))=\sumend(\log(\si_l))$, 
and $\si_l$ corresponds to some $\taucp{s,t}$.
Define 
\begin{equation}\label{lamlogit}
  \si_{l+1}:=\left\{\begin{array}{ll}
                      \sumend(\la^{\sipr_l}_{\si_l}) & \mbox{ if }\si_l\in\Ez^{>\sipr_l}\\[1mm]
                      \sumend(\log(\si_l)) & \mbox{ otherwise,}
                    \end{array}\right.
\end{equation}
which in the former case either corresponds to $\taucp{s,t+1}$ if $\si_{l+1}=\mu_{\si_l}\in\Ez^{>\si_l}$ or $\taucp{s+1,1}$ otherwise,
and in the latter case corresponds to $\taucp{s+1,1}$.
Finally, $\bs_{l+1}$ is defined by 
\[\bs_{l+1}:=\left\{\begin{array}{ll}
                      {\bspr_{l+1}}^\frown\si_{l+1} & \mbox{ if }\si_{l+1}\in\Ez^{>\sipr_{l+1}}\\[1mm]
                      \bspr_{l+1} & \mbox{ otherwise.}
                    \end{array}\right.\]
\end{defi}
\begin{rmk} From the $k+1$-th component ($k+2$-th in Case 2 for $m_n>1\andsp\tau\in\Ez^{>\taupr}$) on $\sivec$ is the sequence of least additive components of indices of $\alvecpl$, 
starting with the terminal index of $\alvecpr$, after omitting superfluous $\nu$-indices, i.e.\ a sub-maximal 
$\nu$-index at the beginning of the maximal extension or indices of a form $\mu_\tau$ 
that are followed in $\alvecpl$ by $\la_\tau$, cf.\ clause 2.2.1 of Definition 5.2 of \cite{CWc}.
All relevant context information regarding $\tauvec$-indices and units/bases in the sense of Definition 5.1 of \cite{CWc} is kept for 
later reference, which motivates the first $k+1$ components of $\sivec$.
\end{rmk}

Now we arithmetically characterize the sequence defined above and establish its generation by term decomposition, which in turn
provides the motivation for the indicator function $\chi$ from Definition 3.1 of \cite{CWc}.
\begin{defi}
Let $\tauvec=(\tau_1,\ldots,\tau_n)\in\RS$ where $n\ge 1$ and $\al\in\Ttvec\cap\Hz\cap\Om_1$.
We define a sequence $\mqtvec(\al)$ of subterms as follows.\index{$\mq$}
\begin{enumerate} 
  \item If $\al\le\tau_1$,\\
        $\mqtvec(\al):=(\al)$.
  \item Suppose that $\al$ is of a form $\thtti(\De+\eta)>\tau_1$ where $1\le i\le n$ and $\eta<\Om_1\mid\De$.\\ 
    Set $\De_0:=\sumend(\De)$ and $\eta_0:=\sumend(\eta)$.
    \begin{enumerate}
      \item[2.1.] Suppose $\eta=\sup_{\si<\eta}\thtti(\De+\si)$ or $\log(\eta_0)=0$. 
        \begin{enumerate}
          \item[2.1.1.] If $\De>0$,\\
                        $\mqtvec(\al):=(\al)^\frown\mq^{{\tauvec_{\restriction_i}}^\frown\al}(\iotial(\De_0))$.
          \item[2.1.2.] If $\De=0$ and $\eta>0$,\\
                        $\mqtvec(\al):=(\al,1)$.
          \item[2.1.3.] Otherwise\\
                        $\mqtvec(\al):=(\al,0)$. 
        \end{enumerate}
      \item[2.2.] Otherwise. 
        \begin{enumerate}
          \item[2.2.1.] If $\De=0$ or $\De>0$ and $\eta_0\in\Ez$,\\
                        $\mqtvec(\al):=(\al)^\frown\mq^{\tauvec_{\restriction_i}}(\eta_0)$. 
          \item[2.2.2.] If $\De>0$ and $\eta_0=\thttj(\rho)$ where $1\le j\le i$ 
                        such that $\rho<\Om_1$, $\rho\not\in\Ez^{>\tau_j}$, and $\logend(\rho)>0$,\\
                        $\mqtvec(\al):=(\al)^\frown\mq^{\tauvec_{\restriction_i}}(\sumend(\rho))$.
          \item[2.2.3.] If $\De>0$ and $\eta_0<\tau_1$,\\
                        $\mqtvec(\al):=(\al)^\frown(\sumend(\log(\eta_0)))$. 
          \item[2.2.4.] Otherwise\\
                        $\mqtvec(\al):=(\al,1)$.
        \end{enumerate}
    \end{enumerate}
\end{enumerate}
\end{defi}

\begin{lem} Let $\tauvec=(\tau_1,\ldots,\tau_n)\in\RS$ where $n\ge 1$ and $\al\in\Ttvec\cap\Hz\cap\Om_1$.
Let $k\in\{0,\ldots,n\}$ be maximal such that $\tau_k\le\al$ where $\tau_0:=1$.
We have $\mqtvec(\al)=(\al)$ if $\al\le\tau_1$ and $\mqtvec(\al)=(\al,0)$ if $\al=\tau_k$, $k>1$.
If $\al\in\Ez^{>\tau_k}$ then $\mqtvec(\al)$ appends to $\al$ the sequence $\mq^{\tauvec_{\restriction_k}}(\sumend(\la^{\tau_k}_\al))$,
and otherwise it appends the sequence $\mq^{\tauvec_{\restriction_k}}(\sumend(\log(\al)))$.
\end{lem}
\begin{proof} This follows directly from the definition. \qed \end{proof}

\begin{lem}\label{simqchilem} In the context of Definition \ref{sivecdefi} let $\sipr:=\sipr_{k+1}$ and $\si:=\si_{k+1}$. We have 
\[\mq^{(\sipr)}(\si)=(\si_{k+1},\ldots,\si_l)\]
where $l$ is minimal such that $\si_l\le\sipr$.
$\chi^{\sipr}$ is constant on $(\si_{k+1},\ldots,\si_l)$ and equal to $1$ if and only if $\si_l=\sipr$.
\end{lem}
\begin{proof} This is a direct consequence of the previous lemma and the definitions involved, cf.\ (\ref{lamlogit}),
using Lemma 7.7 of \cite{W07a} and Lemma 3.3 of \cite{CWc}. \qed \end{proof}
\begin{rmk} As a consequence of the above lemma we obtain an explicit proof of part b) of Lemma 5.5, using the proof of part a), 
and of Corollary 5.6 of \cite{CWc} since the sequences of terms concerned as well as, modulo translations back into $\T^{\sipr}$, 
the corresponding evaluations of the indicator function $\chi$ are matched (skipping intermediate evaluation steps). 
The above preparation also proves useful in the context of Definition \ref{pitcdefi}. 
\end{rmk}

Note that it is possible to translate all relativized terms in a setting $\Ttvec$, see Definition \ref{Ttauvec}, 
back into $\T^{\tau_1}$ or even $\T$, see Definition 6.2
\footnote{In the definition of $\ttau$ there, it should read \emph{For $\xi\in\Tal\cap\Om_2$ in order to define 
$\thtal(\xi)^\ttau$ we distinguish between four cases: [...]}.}
of \cite{W07a}, and establish correspondences between all relevant
subterms. Instead, we have chosen to establish all required invariance properties of operators such as the $\la$-operator with respect
to changes of relativization as in Lemma 7.7 of \cite{W07a}. This could be systematically studied starting from a mapping that 
assigns $\Ttvec$-terms, where $\tauvec\in\RS$, to all $\thti$-subterms of a $\thtnod$-term, where the varying settings of relativization 
given by $\tauvec$ appropriately match the respective nestings of functions $\thte,\ldots,\thti$.

\subsection{Arithmetical characterization of $\Ctwo$}
Subsection 5.3 of \cite{W} provides a detailed picture of the restriction of $\Rtwo$ to $\oneinf$ on the basis of the 
results of \cite{CWc} and displays an elementary recursive arithmetical characterization of this 
structure, given in terms of tracking chains, which we will refer to as $\Ctwo$. We quote this part of \cite{W} in order to increase the
accessibility of \cite{CWc} and the present article: 
 
We begin with a few observations that follow from the results in Section 7 of \cite{CWc} and explain the concept of tracking chains.
Evaluations of all initial chains of some tracking chain $\alvec\in\TC$ form a $\lo$-chain.
Evaluations of initial chains $\alvec_{\restriction_{i,j}}$ where $(i,j)\in\dom(\alvec)$ and $j=2,\ldots,m_i$ with fixed index
$i$ form $\ktwo$-chains. 
Recall that indices $\alcp{i,j}$ are $\ka$-indices for $j=1$ and $\nu$-indices otherwise, see Definitions 5.1 and 5.9 of \cite{CWc}.

According to Theorem 7.9 of \cite{CWc}, an ordinal $\al<\oneinf$ is $\leo$-minimal if and only if its tracking chain consists of 
a single $\ka$-index, i.e.\ if its tracking chain $\alvec$ satisfies $(n,m_n)=(1,1)$. Clearly, the least $\leo$-predecessor of
any ordinal $\al<\oneinf$ with tracking chain $\alvec$ is $\ov(\alvec_{\restriction_{1,1}})=\ka_\alcp{1,1}$.
According to Corollary 7.11 of \cite{CWc} the ordinal $\oneinf$ is $\leo$-minimal.
An ordinal $\al>0$ has a non-trivial $\leo$-reach if and only if $\taucp{n,1}>\taunstar$, hence in particular when $m_n>1$, 
cf.\ condition 2 in Definition 5.1 of \cite{CWc}. 

We now turn to a characterization of the \emph{greatest immediate $\leo$-successor}, $\gs(\al)$, \index{$\gs$}
of an ordinal $\al<\oneinf$ with tracking chain $\alvec$.
Recall the notations $\rho_i$ and $\alvec[\xi]$ from Definition 5.1 of \cite{CWc}. The largest $\al$-$\leo$-minimal ordinal is the root 
of the $\la$th $\al$-$\leo$-component for $\la:=\rho_n\minusp1$. Therefore, if $\al$ has a non-trivial $\leo$-reach, its greatest immediate 
$\leo$-successor $\gs(\al)$ has the tracking chain $\alvec^\frown(\la)$, \emph{unless} either
$\taucp{n,m_n}<\mu_{\taunpr}\andsp\chi^\taunpr(\taucp{n,m_n})=0$, where 
$\tc(\gs(\al))=\alvec[\alcp{n,m_n}+1]$, or $\alvec^\frown(\la)$ is in conflict with either condition 5 of Definition 5.1 of \cite{CWc},
in which case $\tc(\gs(\al))={\alvec_{\restriction_{n-1}}}^\frown(\alcp{n,1},\ldots,\alcp{n,m_n},1)$, or condition 6
of Definition 5.1 of \cite{CWc}, in which case $\tc(\gs(\al))=\alvec_{\restriction_{i,j+1}}[\alcp{i,j+1}+1]$.
\footnote{This condition is missing in \cite{W}.}
In case $\al$ does not have any $\lo$-successor, we set $\gs(\al):=\al$.

$\al$ is $\letwo$-minimal if and only if for its tracking chain $\alvec$ we have $m_n\le2$ and $\taunstar=1$,
and $\al$ has a non-trivial $\letwo$-reach if and only if $m_n>1$ and $\taucp{n,m_n}>1$. Note that any $\al\in\On$ with
a non-trivial $\letwo$-reach is the proper supremum of its $\lo$-predecessors, hence $\oneinf$ does not possess any $\ktwo$-successor. 
Iterated closure under the relativized notation system $\Tt$ for $\tau=\oneinf, (\oneinf)^\infty, \ldots$ results in the infinite $\ktwo$-chain
through $\On$. Its $\lo$-root is $\oneinf$, the root of the \emph{master main line}\index{master main line} 
of $\Rtwo$, outside the core of $\Rtwo$, i.e.\ $\oneinf$, see \cite{W3}.

Recall Definition 7.7 of \cite{CWc} of $\predec_i$ and $\predecs_i$.\index{$\predec_i$, $\predecs_i$}
According to part (a) of Theorem 7.9 of \cite{CWc} $\al$ has a greatest $\lo$-predecessor if and only if it is not
$\leo$-minimal and has a trivial $\letwo$-reach (i.e.\ does not have any $\ktwo$-successor). This is the case if 
and only if either $m_n=1$ and $n>1$, where we have $\predec_1(\al)=\ordcp{n-1,m_{n-1}}(\alvec)$, or $m_n>1$ and
$\taucp{n,m_n}=1$. In this latter case $\alcp{n,m_n}$ is of a form $\xi+1$ for some $\xi\ge 0$, and using again the notation
from Definition 5.1 of \cite{CWc} we have
$\predec_1(\al)=\ov(\alvec[\xi])$ if $\chi^\taucp{n,m_n-1}(\xi)=0$, whereas $\predec_1(\al)=\ov(\me(\alvec[\xi]))$ in
the case $\chi^\taucp{n,m_n-1}(\xi)=1$.
  
Recall Definition 7.12 of \cite{CWc}, defining for $\alvec\in\TC$ the notation $\alvec^\star$\index{$\alvec^\star$} 
and the index pair $\gbo(\alvec)=:(n_0,m_0)$,\index{$\gbo$} 
which according to Corollary 7.13 of \cite{CWc} enables us to express the $\leo$-reach $\lh(\al)$\index{$\lh$} of $\al:=\ov(\alvec)$, 
cf.\ Definition 7.7 of \cite{CWc}, by
\begin{equation}\label{lhequation}\lh(\al)=\ov(\me(\bevec^\star)),
\end{equation} 
where $\bevec:=\alvec_{\restriction_{n_0,m_0}}$, which in the case $m_0=1$ is equal to 
$\ov(\me(\bevec))=\ordcp{n_0,1}(\alvec)+\dpf_{\tauticp{n_0,0}}(\taucp{n_0,1})$
and in the case $m_0>1$ equal to $\ov(\me(\bevec[\mu_\taucp{n_0,m_0-1}]))$.
Note that if $\cml(\alvec^\star)$ does not exist we have
\[\lh(\al)=\ov(\me(\alvec^\star)),\]
and the tracking chain $\bevec$ of any ordinal $\be$ such that $\ov(\alvec^\star)\leo\be$ is then an extension of $\alvec$,
$\alvec\sub\bevec$, as will follow from Lemma \ref{subresplem}.

The relation $\leo$ can be characterized by 
\begin{equation}\label{leoequation}\al\leo\be 
\quad\aeq\quad\al\le\be\le\lh(\al),
\end{equation} 
showing that $\leo$ is a forest contained in $\le$ that \emph{respects} the ordering $\le$: 
if $\al\le\be\le\ga$ and $\al\leo\ga$ then $\al\leo\be$. 
  
We now recall how to retrieve the greatest $\ktwo$-predecessor of an ordinal below $\oneinf$, if it exists, and the iteration of this procedure to obtain the maximum
chain of $\ktwo$-predecessors. Recall Definition 5.3 and Lemma 5.10 of \cite{CWc}.
Using the following proposition we obtain two other useful characterizations of the relationship $\al\letwo\be$. 
 
\begin{prop}[5.6 of \cite{W}]\label{letwopredprop} Let $\al<\oneinf$ with $\tc(\al)=:\alvec$.
We define a sequence $\sivec\in\RS$ as follows.
\begin{enumerate}
\item If $m_n\le2$ and $\taunstar=1$, whence $\al$ is $\letwo$-minimal according to Theorem 7.9 of \cite{CWc}, set $\sivec:=()$. Otherwise,
\item if $m_n>2$, whence $\predec_2(\al)=\ordcp{n,m_n-1}(\alvec)$ with base $\taucp{n,m_n-2}$ by Theorem 7.9 of \cite{CWc}, we set
$\sivec:=\cs(\alvec_{\restriction_{n,m_n-2}})$,
\item and if $m_n\le2$ and $\taunstar>1$, whence $\predec_2(\al)=\ordcp{i,j+1}(\alvec)$ with base $\taucp{i,j}$ where $(i,j):=n^\star$, again according to Theorem 7.9 of \cite{CWc}, we set
$\sivec:=\cs(\alvec_{\restriction_{i,j}})$.
\end{enumerate}
Each $\si_i$ is then of a form $\taucp{k,l}$ where $1\le l< m_k$, $1\le k \le n$. The corresponding $\ktwo$-predecessor of $\al$ is $\ordcp{k,l+1}(\alvec)=:\be_i$.
We obtain sequences $\sivec=(\si_1,\ldots,\si_r)$ and $\bevec=(\be_1,\ldots,\be_r)$ with $\be_1\ktwo\ldots\ktwo\be_r\ktwo\al$, where $r=0$ if $\al$ is $\letwo$-minimal,
so that $\predecs_2(\al)=\{\be_1,\ldots,\be_r\}$ and hence
$\be\ktwo\al$ if and only if $\be\in\predecs_2(\al)$, displaying that $\letwo$ is a forest contained in $\leo$.\qed 
\end{prop}

\begin{lem}[5.7 of \cite{W}]\label{letwosuclem} Let $\al,\be<\oneinf$ with tracking chains $\tc(\al)=\alvec=(\alvec_1,\ldots,\alvec_n)$, where
$\alvec_i=(\alcp{i,1},\ldots,\alcp{i,m_i})$, $1\le i\le n$, and $\tc(\be)=\bevec=(\bevec_1,\ldots,\bevec_l)$, 
$\bevec_i=(\becp{i,1},\ldots,\becp{i,k_i})$, $1\le i\le l$. 
Assume further that $\alvec\subseteq\bevec$ with
associated chain $\tauvec$ and that $m_n>1$. Set $\tau:=\taucp{n,m_n-1}$. The following are equivalent:
\begin{enumerate}
\item $\al\letwo\be$
\item $\tau\le\taucp{j,1}$ for $j=n+1,\ldots,l$
\item $\tauti\mid\be$.
\end{enumerate}
\end{lem}
\begin{proof} The proof is given in detail in \cite{W}.
\qed \end{proof}

Applying the elementary recursive mappings $\tc$, see Section 6 of \cite{CWc}, and $\ov$, we are now able to formulate the arithmetical 
characterization of $\Ctwo$.

\begin{cor}[5.8 of \cite{W}]\label{elemreccharcor} The structure $\Ctwo$ is characterized elementary recursively by
\begin{enumerate}
\item $(\oneinf,\le)$ is the standard ordering of the classical notation system $\oneinf=\Targ{1}\cap\Om_1$, see \cite{W07a},
\item $\al\leo\be$ if and only if $\al\le\be\le\lh(\al)$ where $\lh$ is given by equation \ref{lhequation}, and
\item $\al\letwo\be$ if and only if $\tc(\al)\subseteq\tc(\be)$ and condition 2 of Lemma \ref{letwosuclem} holds.\qed
\end{enumerate}
\end{cor}

Recall Definition 5.13 of \cite{CWc} which characterizes the standard linear ordering $\le$ on $\oneinf$
by an ordering $\letc$ on the corresponding tracking chains.
We can formulate a characterization of the relation $\leo$ (below $\oneinf$) in terms of the corresponding 
tracking chains as well. This follows from an inspection of the ordering $\letc$ in combination with the above 
statements. Let $\al,\be<\oneinf$ with tracking chains $\tc(\al)=\alvec=(\alvec_1,\ldots,\alvec_n)$, 
$\alvec_i=(\alcp{i,1},\ldots,\alcp{i,m_i})$, $1\le i\le n$, and $\tc(\be)=\bevec=(\bevec_1,\ldots,\bevec_l)$, 
$\bevec_i=(\becp{i,1},\ldots,\becp{i,k_i})$, $1\le i\le l$. 
We have $\al\leo\be$ if and only if either $\alvec\subseteq\bevec$ or there exists 
$(i,j)\in\dom(\alvec)\cap\dom(\bevec)$, $j<\min\{m_i,k_i\}$, such that 
$\alvec_{\restriction_{i,j}}=\bevec_{\restriction_{i,j}}$ and $\alcp{i,j+1}<\becp{i,j+1}$,
and we either have $\chi^\taucp{i,j}(\alcp{i,j+1})=0\andsp(i,j+1)=(n,m_n)$ or 
$\chi^\taucp{i,j}(\alcp{i,j+1})=1\andsp\al\leo\ov(\me(\alvec_{\restriction_{i,j+1}}))\lo\be$.
Iterating this argument and recalling Lemma 5.5 of \cite{CWc} we obtain the following paramount

\begin{prop}[5.10 of \cite{W}]\label{gbocharprop} Let $\al$ and $\be$ with tracking chains $\alvec$ and $\bevec$, respectively, as above.
We have $\al\leo\be$ if and only if either $\alvec\subseteq\bevec$ or
there exists the $\klex$-increasing chain of index pairs $(i_1,j_1+1),\ldots,(i_s,j_s+1)\in\dom(\alvec)$ of maximal length $s\ge 1$ where $j_r\ge 1$ for $r=1,\ldots,s$,
such that $(i_1,j_1+1)\in\dom(\bevec)$, $\alvec_{\restriction_{i_1,j_1}}=\bevec_{\restriction_{i_1,j_1}}$, $\alcp{i_1,j_1+1}<\becp{i_1,j_1+1}$,
\[\alvec_{\restriction_{i_s,j_s+1}}\subseteq\alvec\subseteq\me(\alvec_{\restriction_{i_s,j_s+1}}),\]
and $\chi^\taucp{i_r,j_r}(\alcp{i_r,j_r+1})=1$ at least whenever $(i_r,j_r+1)\not=(n,m_n)$. 
Setting $\al_r:=\ov(\alvec_{\restriction_{i_r,j_r+1}})$ for $r=1,\ldots,s$ as well as $\al^+_r:=\ov(\me(\alvec_{\restriction_{i_r,j_r+1}}))$ 
for $r$ such that $\chi^\taucp{i_r,j_r}(\alcp{i_r,j_r+1})=1$ and $\al^+_s:=\al$ if $\chi^\taucp{i_s,j_s}(\alcp{i_s,j_s+1})=0$ we have
\[\al_1\ktwo\ldots\ktwo\al_s\letwo\al\leo\al^+_s\lo\ldots\lo\al^+_1\lo\ov(\bevec_{\restriction_{i_1,j_1+1}})\leo\be.\]
For $\be=\lh(\al)$ the cases $\alvec\subseteq\bevec$ and $s=1$ with $(i_1,j_1+1)=(n,m_n)$ correspond to the situation
$\gbo(\alvec)=(n,m_n)$, while otherwise we have $\gbo(\alvec)=(i_1,j_1+1)$.\qed
\end{prop}
\begin{rmk}[\cite{W}]\label{intermedextsteprmk}
Note that the above index pairs characterize the relevant sub-maximal $\nu$-indices in the initial chains of $\alvec$
with respect to $\bevec$ and omit the intermediate steps of maximal ($\me$-) extension along the iteration.
Using Lemma 5.5 of \cite{CWc} we observe that the sequence $\taucp{i_1,j_1},\ldots,\taucp{i_s,j_s}$ of bases in the above proposition satisfies
\begin{equation}\label{baseseqeq}
\taucp{i_1,j_1}<\ldots<\taucp{i_s,j_s}\quad\mbox{ and }\quad\taucp{i_s,j_s}<\taucp{i,1} \quad\mbox{ for every }i\in(i_s,n],
\end{equation}
so that in the case where $\al\lo\be$ and $\alvec\not\subseteq\bevec$ we have $\al\lo\gs(\al)\leo\be$ with $\taucp{i_s,j_s}\mid\rho_n\minusp1$.
\end{rmk}
\begin{lem}[5.11 of \cite{W}]\label{subresplem}
The relation $\sub$ of initial chain on $\TC$ respects the ordering $\letc$ and hence also the characterization of $\leo$ on $\TC$.
\end{lem}
\begin{proof} This is a consequence of the above Proposition \ref{gbocharprop}. For details see \cite{W}.
\qed \end{proof}

While it is easy to observe that in $\Rtwo$ the relation $\leo$ is a forest that respects $\le$ and the relation $\letwo$ is a forest contained in $\leo$
which respects $\leo$, we can now conclude that this also holds for the arithmetical formulations of $\leo$ and $\letwo$ in $\Ctwo$, 
without referring to the results in Section 7 of \cite{CWc}.  

\begin{cor}[5.14 of \cite{W}] Consider the arithmetical characterizations of $\leo$ and $\letwo$ on $\oneinf$.
The relation $\letwo$ respects $\leo$, i.e.\ whenever $\al\leo\be\leo\ga<\oneinf$ and $\al\letwo\ga$, then $\al\letwo\be$.
\end{cor}
\begin{proof} In the case $\bevec\subseteq\gavec$ this directly follows from Lemma \ref{letwosuclem}, 
while otherwise we additionally employ Proposition \ref{gbocharprop} and property \ref{baseseqeq}.
\qed \end{proof}

We conclude this section with a characterization of $\lh_2$ based on Proposition \ref{letwosuclem} that is not given in \cite{W}.
In short, the following proposition formalizes the procedure of extending $\alvec$ stepwise maximally by tracking chains with the 
modification that extending $\ka$-indices must be proper multiples of $\taupr$ as specified below. 

\begin{prop} Let $\al<\oneinf$ and write $\tc(\al)=\alvec=(\alvec_1,\ldots,\alvec_n)$, 
$\alvec_i=(\alcp{i,1},\ldots,\alcp{i,m_i})$, $1\le i\le n$, with associated chain $\tauvec$. 
If $m_n=1$ or $\taucp{n,m_n}=1$, we have $\lh_2(\al)=\al$ as this is the
already mentioned characterization of trivial $\letwo$-reach. \index{$\lh$!$\lh_2$}
Now suppose that $m_n>1$ and $\tau:=\taucp{n,m_n}>1$, set $\taupr:=\taucp{n,m_n-1}$.
\\[2mm]
{\bf Case 1:} $\chi^\taupr(\tau)=0$. Let $\bevec:=\alvec$ and 
$\bevecpl:={\alvec_{\restriction_{n-1}}}^\frown(\alcp{n,1},\ldots,\alcp{n,m_n},\mu_\tau)$ if $\tau\in\Ez^{>\taupr}$, otherwise
$\bevecpl:=\alvec^\frown(\rhoargs{\taupr}{\tau})$.
\\[2mm]
{\bf Case 2:} $\chi^\taupr(\tau)=1$. Define $\bevec:=\me(\alvec)$, 
extending $\alvec$ by vectors $\bevec_i=(\becp{i,1},\ldots,\becp{i,k_i})$, $1\le i\le l$, and
accordingly for the associated chain $\tauvec$. 
According to Corollary 5.6 of \cite{CWc} the extending $\ka$-index of $\ec(\bevec)$ is of the
form $\xi+\taupr$ for suitable minimal $\xi\ge 0$. $\lh_2(\al)$ is then equal to $\ov(\alvec[\alcp{n,m_n}+1])\minusp\tautipr$. 
If $\xi=0$, we have $\lh_2(\al)=\ov(\bevec)$,
otherwise define $\bevecpl:=\bevec^\frown(\xi)$ if this satisfies condition 5
of Definition 5.1 of \cite{CWc} while $\bevecpl:={\bevec_{\restriction_{l-1}}}^\frown(\becp{l,1},\ldots,\becp{l,k_l},\mu_\taucp{l,k_l})$
otherwise.
\\[2mm]
In case it is defined, starting from $\bevecpl$ we iterate the following procedure. Let $\gavec$ be the tracking chain reached so far.
Consider the maximal $1$-step extension $\gavec^+$ of $\gavec$. If this adds a $\nu$-index, continue with $\gavec^+$.
Otherwise let $\eta=\etapr+\eta_0$ be the extending $\ka$-index, where $\taupr\mid\etapr$ and $\eta_0<\taupr$.
If $\etapr=0$, we have $\lh_2(\al)=\ov(\gavec)$. Now suppose that $\etapr>0$.
If $\gavec^\frown(\etapr)\in\TC$, continue with that tracking chain, otherwise $\gavec$ is of a form 
$\gavecpr^\frown(\gacp{h,1},\ldots,\gacp{h,r})$ with associated chain $\sivec$ that satisfies $\si:=\sicp{h,r}\in\Ez^{>\sipr_h}$,
and we continue with $\gavecpr^\frown(\gacp{h,1},\ldots,\gacp{h,r},\mu_\si)$. \qed   
\end{prop}

\section{Spanning and closed sets of tracking chains}\label{spcltcsec}

The notion of closedness for sets of tracking chains is central to the investigation
of the core of $\Rtwo$, as it is crucial for isominimal realization.
In preparation for a relativized notion of closedness, we will first introduce sets of tracking chains that are spanning above some 
given tracking chain $\alvec$, considerably extending sets of tracking chains that are weakly spanning above $\alvec$ according to 
Definition 5.3 of \cite{W}.
For the reader's convenience, we begin with a review of the preparations made in \cite{W}, Subsections 5.2 and 5.3, which provide 
a generalization of the notion of maximal extension $\me$.

\begin{defi}[Pre-closedness, Def.\ 5.1 of \cite{W}]\label{precldefi} Let $M\finsub\TC$. $M$ is \emph{pre-closed}\index{pre-closed} if and only if $M$ 
\begin{enumerate}
\item is \emph{closed under initial chains:} if $\alvec\in M$ and $(i,j)\in\dom(\alvec)$ then $\alvec_{\restriction_{(i,j)}}\in M$,
\item is \emph{$\nu$-index closed:} if $\alvec\in M$, $m_n>1$, $\alcp{n,m_n}=_\ANF\xi_1+\ldots+\xi_k$ then
$\alvec[\mu_{\taupr}], \alvec[\xi_1+\ldots+\xi_l]\in M$ for $1\le l\le k$,
\item \emph{unfolds minor $\letwo$-components:} if $\alvec\in M$, $m_n>1$, and $\tau<\mu_\taupr$ then: 
\begin{enumerate}
\item[3.1.] ${\alvec_{\restriction_{n-1}}}^\frown(\alcp{n,1},\ldots,\alcp{n,m_n},\mu_\tau)\in M$ in the case $\tau\in\Ez^{>\taupr}$, and
\item[3.2.] $\alvec^\frown(\varrho^\taupr_\tau)\in M$ otherwise, provided that $\varrho^\taupr_\tau>0$, 
\end{enumerate}
\item is \emph{$\ka$-index closed:} if $\alvec\in M$, $m_n=1$, and $\alcp{n,1}=_\ANF\xi_1+\ldots+\xi_k$, then:
\begin{enumerate}
\item[4.1.] ${\alvec_{\restriction_{n-2}}}^\frown(\alcp{n-1,1},\ldots,\alcp{n-1,m_{n-1}},\mu_{\xi_1})\in M$ in the case 
$m_{n-1}>1$ $\andsp$ $\xi_1=\taucp{n-1,m_{n-1}}\in\Ez^{>\taucp{n-1,m_{n-1}-1}}$, while otherwise  
${\alvec_{\restriction_{n-1}}}^\frown(\xi_1)\in M$, and
\item[4.2.] ${\alvec_{\restriction_{n-1}}}^\frown(\xi_1+\ldots+\xi_l)\in M$ for $l=2,\ldots,k$,
\end{enumerate}
\item \emph{maximizes $\me$-$\mu$-chains:} if $\alvec\in M$, $m_n\ge 1$, and $\tau\in\Ez^{>\taupr}$, then:
\begin{enumerate}
\item[5.1.] ${\alvec_{\restriction_{n-1}}}^\frown(\alcp{n,1},\mu_\tau)\in M$ if $m_n=1$, and
\item[5.2.] ${\alvec_{\restriction_{n-1}}}^\frown(\alcp{n,1}\ldots,\alcp{n,m_n},\mu_\tau)\in M$ if $m_n>1$ $\andsp$ $\tau=\mu_\taupr=\la_\taupr$. 
\end{enumerate}
\end{enumerate}
\end{defi}
\begin{rmk}[\cite{W}]\label{intermedmurmk}
Pre-closure of some $M\finsub\TC$ is obtained by closing under clauses 1 -- 5 in this order once, hence finite: 
in clause 5 note 
that $\mu$-chains are finite since the $\htarg{}$-measure of terms strictly decreases with each application of $\mu$. Note further that intermediate indices
are of the form $\la_\taupr$, whence we have a decreasing $\operatorname{l}$-measure according to inequality \ref{lalen} following Definition \ref{Ttauvec}.
\end{rmk}
\begin{defi}[Spanning sets of tracking chains, corrected Def.\ 5.2 of \cite{W}]\label{spanningdefi} 
$M\finsub\TC$ is \emph{spanning}\index{spanning} if and only if it is pre-closed and closed under
\begin{enumerate}
\item[6.] \emph{unfolding of $\leo$-components:} for $\alvec\in M$, if $m_n=1$ and $\tau\not\in\Ez^{\ge\taupr}$ 
(i.e.\ $\tau=\taucp{n,m_n}\not\in\Ezone$, $\taupr=\taunstar$), let
\[\log((1/\taupr)\cdot\tau)=_\ANF\xi_1+\ldots+\xi_k,\]
if otherwise $m_n>1$ and $\tau=\mu_\taupr$ such that $\tau<\la_\taupr$ in the case $\tau\in\Ez^{>\taupr}$, let
\[\la_\taupr=_\ANF\xi_1+\ldots+\xi_k.\]
Set $\xi:=\xi_1+\ldots+\xi_k$, unless $\xi>0$ and $\alvec^\frown(\xi_1+\ldots+\xi_k)\not\in\TC$,
\footnote{This is the case if clause 6 of Def.\ 5.1 of \cite{CWc} does not hold.} 
in which case we set $\xi:=\xi_1+\ldots+\xi_{k-1}$.
Suppose that $\xi>0$. Let $\alvecpl$ denote the vector $\{\alvec^\frown(\xi)\}$ if this is a tracking chain, 
or otherwise the vector ${\alvec_{\restriction_{n-1}}}^\frown(\alcp{n,1},\ldots,\alcp{n,m_n},\mu_\taucp{n,m_n})$.
\footnote{This case distinction, due to clause 5 of Def.\ 5.1 of \cite{CWc}, is missing in \cite{W}.}
Then the closure of $\{\alvecpl\}$ under clauses 4 and 5 is contained in $M$.
\end{enumerate}
\end{defi}
\begin{rmk}[\cite{W}]
Closure of some $M\finsub\TC$ under clauses 1 -- 6 is a finite process since pre-closure is finite and since 
the $\ka$-indices added in clause 6 strictly decrease in $\operatorname{l}$-measure.
Semantically, the above notion of spanning sets of tracking chains and closure under clauses 1 -- 6 leaves 
some redundancy in the form that certain
$\ka$-indices could be omitted. This will be addressed later, since the current formulation is advantageous
for technical reasons.
\end{rmk}
\begin{defi}[Relativization, Def.\ 5.3 of \cite{W}]\label{reldefi} Let $\alvec\in\TC\cup\{()\}$ and $M\finsub\TC$ be a set of tracking chains
that properly extend $\alvec$.
$M$ is \emph{pre-closed above $\alvec$}\index{pre-closed!pre-closed above $\alvec$} if and only if it is pre-closed 
with the modification that clauses 1 -- 5 only apply when the respective resulting tracking chains $\bevec$ 
properly extend $\alvec$.
$M$ is \emph{weakly spanning above $\alvec$}\index{spanning!weakly spanning above $\alvec$} if and only if $M$ is pre-closed above 
$\alvec$ and closed under clause 6. 
\end{defi}

\begin{lem}[5.4 of \cite{W}]\label{meclosurelem} If $M$ is spanning (weakly spanning above some $\alvec$), then it is closed under $\me$
(closed under $\me$ for proper extensions of $\alvec$).
\end{lem}
\begin{proof} This follows directly from the definitions involved. \qed \end{proof}

\noindent On the basis of Lemma \ref{meclosurelem}, Equation \ref{lhequation} has the following
\begin{cor}[5.5 of \cite{W}]\label{leoclosurecor} Let $M\finsub\TC$ be spanning (weakly spanning above some $\alvec\in\TC$) and $\bevec\in M$, 
$\be:=\ov(\bevec)$.
Then \[\tc(\lh(\be))\in M,\] provided that $\ov(\bevec_{\restriction_{\gbo(\bevec)}})$ is a proper extension of $\alvec$
in the case that $M$ is weakly spanning above $\alvec$.
\qed
\end{cor}

\begin{cor}[5.9 of \cite{W}]\label{lhtwoclscor}
Let $M\finsub\TC$ be spanning (weakly spanning above some $\alvec\in\TC$). Then $M$ is closed under $\lh_2$.
\end{cor}
\begin{proof} This follows from Lemma \ref{meclosurelem} using Lemma \ref{letwosuclem}, cf.\ Corollaries 5.6 and 7.13 of \cite{CWc}. 
\qed \end{proof}

\noindent The closure under $\lh$ has a convenient sufficient condition on the basis of the following 
\begin{defi}[5.12 of \cite{W}] A tracking chain $\alvec\in\TC$ is called \emph{convex}\index{convex} if and only if every $\nu$-index in $\alvec$
is maximal, i.e.\ given by the corresponding $\mu$-operator.
\end{defi}

\begin{cor}[5.13 of \cite{W}]\label{lhclscor}
Let $\alvec\in\TC$ be convex and $M\finsub\TC$ be weakly spanning above $\alvec$.
Then $M$ is closed under $\lh$.
\end{cor}
\begin{proof}
This is a consequence of Proposition \ref{gbocharprop}, Corollary \ref{leoclosurecor}, Lemma \ref{meclosurelem}, 
and equation \ref{lhequation}.
\qed \end{proof}

This concludes the review of \cite{W} to provide a solid basis for developments to come.
We now introduce some useful notation before relativizing the notion of spanning set of tracking chains.

\begin{defi}\label{Malgsdefi}
Let $M\finsub\TC$.
\begin{enumerate}
\item For $\alvec\in M$ where $\alvec_i=(\alcp{i,1},\ldots,\alcp{i,m_i})$, $1\le i\le n$, 
let $M_\alvec$ denote the subset \index{$M_\alvec$}
\[M_\alvec:=\{\bevec\in M\mid \ov(\alvec)\lo\ov(\bevec)\}.\]
We also sometimes denote a finite set of $\lo$-successors of some $\alvec\in\TC$ by $M_\alvec$, 
i.e.\ a superset $M$ containing $\alvec$ is not required.
For $\alvec\in\TC$ we define $I(\alvec)$\index{$I(\alvec)$} to be the set of all initial chains of $\alvec$, including $\alvec$.  
For convenience we set $I(()):=\emptyset$ and $M_{()}:=M$.
\item Set $\gs(M_\alvec):=0$ if $M_\alvec=\emptyset$, otherwise let $\bevec\in M_\alvec$, $\bevec_i=(\becp{i,1},\ldots,\becp{i,k_i})$, 
$1\le i\le l$, be the unique chain corresponding to the greatest immediate
$\lo$-successor of $\ov(\alvec)$ in $\ov[M_\alvec]$, and let $\sivec$ be the chain associated with $\bevec$.
We define \index{$\gs(M_\alvec)$}
\[\gs(M_\alvec):=\left\{\begin{array}{ll}
            \becp{l,1}&\mbox{if } k_l=1\\[2mm]
            \sicp{l,k_l-1}&\mbox{otherwise.}
            \end{array}\right.\]
If $\bevec$ is a $1$-step extension of $\alvec$, we call $\gs(M_\alvec)$ the \emph{$\ka$-index of the greatest immediate $\lo$-successor 
of $\alvec$ in $M_\alvec$}.
\end{enumerate}
\end{defi}

We now strengthen the notion of weakly spanning sets of tracking chains above some $\alvec$.
Proposition \ref{gbocharprop} will play a central role in the definition of (relativized) spanningness as it
characterizes $\leo$ in terms of tracking chains in the sense that necessary and sufficient
conditions for tracking chains $\alvec,\bevec\in\TC$ to satisfy $\ov(\alvec)\leo\ov(\bevec)$ are given.
By Lemma \ref{subresplem} the relation $\sub$ of initial chain on $\TC$ respects the ordering $\letc$,
hence also the characterization of $\leo$ on $\TC$. 

\begin{defi}\label{branchinginddefi}
According to Proposition \ref{gbocharprop}, for $\alvec,\bevec\in\TC$ such that $\ov(\alvec)\lo\ov(\bevec)$ and $\alvec\not\sub\bevec$,
there exists $(i,j)\in\dom(\alvec)\cap\dom(\bevec)$ such that $\alvec_{\restriction_{i,j}}=\bevec_{\restriction_{i,j}}$ and
$\alcp{i,j+1}<\becp{i,j+1}$.  
We call $(i,j+1)=:\bp(\alvec,\bevec)$ the \emph{branching index pair of $\alvec$ and $\bevec$}.\index{branching index pair,$\bp$} 
If the above conditions on 
$\alvec$ and $\bevec$ do not hold, we say that the branching index pair of $\alvec$ and $\bevec$ does not exist.
\end{defi}

\begin{defi}[Spanning sets of tracking chains above $\alvec$]\label{spanningabovedefi}
Let $\alvec\in\TC$ and $M_\alvec\finsub\TC$ be as in the above definition.
$M_\alvec$ is called \emph{spanning above $\alvec$}\index{spanning!spanning above $\alvec$} 
if it is closed under clauses 1 -- 6 of Definitions \ref{precldefi} and \ref{spanningdefi}  
with the modification that the resulting respective tracking chains $\bevec$ satisfy $\ov(\alvec)\lo\ov(\bevec)$, and if it 
\begin{enumerate}
\item[7.] \emph{supplies implicit maximal extensions:} For any $\bevec\in M_\alvec$ such that $\bp(\alvec,\bevec)=:(i,j+1)$ exists
with $\chi^\taucp{i,j}(\taucp{i,j+1})=1$ (where $\tauvec$ is the chain associated with $\alvec$), we have 
$\alvecpr:=\me(\alvec_{\restriction_{i,j+1}})\in M_\alvec$, provided that $\ov(\alvec)\lo\ov(\alvecpr)$, and
\item[8.] \emph{extends main lines:} if $\cml(\bevec)=:(i,j)$ exists for some $\bevec\in M_\alvec$, then 
$\bevec_{\restriction_{i,j+1}}[\mu_\sicp{i,j}]\in M_\alvec$ where $\sivec$ is the chain associated with $\bevec$. 
\end{enumerate}
For $\alvec=()$ any spanning set of tracking chains according to Definition \ref{spanningdefi} is called \emph{spanning above $\alvec$}.
\end{defi}
\begin{rmk} Any $M\finsub\TC$ that is spanning according to Definition \ref{spanningdefi}, is closed under clauses 7 and 8; 
hence closure under clauses 1 -- 8 is a finite process.
\end{rmk}
\begin{lem}\label{strongspanlem}
Let $M$ be spanning above some $\alvec\in\TC$. Then $M$ is closed under $\me$, $\lh$, and $\lh_2$. If $\alvec$ is convex, then
every $\bevec\in M$ is a proper extension of $\alvec$, i.e.\ $\alvec\sub\bevec$. Thus, for convex $\alvec$, $M$ is spanning above $\alvec$
if and only if it is weakly spanning above $\alvec$ according to Definition \ref{reldefi}. 
\end{lem}
\begin{proof} Lemma \ref{meclosurelem} yields the claim regarding $\me$. 
Condition 8 above in conjunction with condition 2 of Definition \ref{precldefi} and characterization (\ref{lhequation}) of $\lh$
then imply the claim regarding $\lh$. And the claim regarding $\lh_2$ follows from Corollary \ref{lhtwoclscor}. 
If $\alvec$ is convex, conditions 7 and 8 never apply to $(i,j)$ such that $(i,j+1)\in\dom(\alvec)$.
\qed \end{proof}

For given $\al<\oneinf$, the following proposition characterizes the ordinals $\be$ such that $\al\lo\be$ 
and there does not exist any $\ga$ with $\al<\ga<_2\be$ in terms of tracking chains, cf.\ Theorem 7.9 of \cite{CWc} and
Proposition \ref{letwopredprop}.

\begin{prop}[Relative $\le_2$-minimality]\label{relletwoprop} Let $\alvec,\bevec\in\TC$ satisfy $\al:=\ov(\alvec)\lo\ov(\bevec)=:\be$. 
Let $\sivec$ be the chain associated with $\bevec$, $\bevec_i=(\becp{i,1},\ldots,\becp{i,k_i})$, $i=1,\ldots,l$.
According to Theorem 7.9 of \cite{CWc} $\be$ is $\le_2$-minimal if and only if $k_l\le2$ and $\si^\star_l=1$. 
In the case $\alvec\not\sub\bevec$ let $(i,j+1):=\bp(\alvec,\bevec)$, otherwise let $(i,j):=(n,m_n)$.
Then \[\be \mbox{ is $\al$-$\le_2$-minimal if and only if either $(i,j+1)=(l,k_l)$ or $k_l\le2\andsp l^\star\klex(i,j)$.}\] 
\end{prop}
\begin{proof} Cf.\ Theorem 7.9 of \cite{CWc} and Proposition \ref{letwopredprop}.
\qed \end{proof}

The following definition and theorem give a flavor of the expressive power of tracking chains in the sense that isomorphisms of
intervals with the same $\ktwo$-predecessors can be identified easily.

\begin{defi}[Vertical translation]\label{verttrandefi}\index{vertical translation} Let $\alvec\in\TC\cup\{()\}$, where
$\alvec_i=(\alcp{i,1},\ldots,\alcp{i,m_i})$, $1\le i\le n$, and $n\ge 0$. Let $M=M_\alvec\finsub\TC$ be a set of proper extensions of $\alvec$ 
of the form $M=\{\bevec\}\cup M_\bevec$ for a tracking chain $\bevec$,
where $\bevec_i=(\becp{i,1},\dots,\becp{i,k_i})$, $1\le i\le l$, such that $k_l=1$ and $l^\star=(1,0)$ if $\alvec=()$ while 
$l^\star\klex(n,m_n)$ otherwise, so that by Proposition \ref{gbocharprop} and (\ref{baseseqeq}) each $\gavec\in M$ is an 
extension of $\bevec$.
For $\gavec\in M$ we define the tracking chain $\gavecpr$ by
\begin{equation}\label{verttraneq}\gavec^\prime_{i,j}:=\left\{\begin{array}{ll}
            \alcp{i,j}&\mbox{if } (i,j)\in\dom(\alvec)\\[2mm]
            \taucp{l,1}&\mbox{if } (i,j)=(n+1,1)\\[2mm]
            \gacp{l-n-1+i,j}&\mbox{if } (n+1,1)\klex(i,j)\andsp(l-n-1+i,j)\in\dom(\gavec).
            \end{array}\right.
\end{equation}
We define $M^\prime:=\{\gavecpr\mid\gavec\in M\}$.
\end{defi}

\begin{theo}\label{verttrantheo} Let $M,\alvec,\bevec$ be as in the above definition and $I:=I(\alvec)$ as in Definition \ref{Malgsdefi}. 
Then $M^\prime$ consists of tracking chains that properly extend $\alvec$, 
and the images $\ov[I\cup M]$ and $\ov[I\cup M^\prime]$ are isomorphic substructures of $\Ctwo$, both closed under $\ktwo$-predecessors.
\end{theo}
\begin{proof} The claims are verified by close inspection of the definitions involved.
Notice that $\bevecpr=\alvec^\frown(\taucp{l,1})$ is a tracking chain since our assumptions
prevent a violation of condition 5 in Definition 5.1 of \cite{CWc} and 
imply that for all $(r,s)\in\dom(\bevec)$ such that $l^*=:(i,j)\klex(r,s)\klex(l,1)$ we have 
$\taucp{l,1}<\rho_r(\bevec_{\restriction_{r,s}})$. 
\qed \end{proof}
\begin{rmk} Note that in the case where $\cml(\bevec)=:(i,j)$ exists, 
the isomorphic copy $M^\prime$ might lose $\lo$-connections up to $\bevec^+:=\bevec_{\restriction_{i,j+1}}[\mu_{\taucp{i,j}}]$.
\end{rmk}
\begin{defi}\label{Mdbdefi}
Let $M\finsub\TC$. We will make use of the notation $M_\alvec$ as in Definition \ref{Malgsdefi}.
\begin{enumerate}
\item\label{principalchainpart} Suppose $\alvec\in\TC$, $\alvec_i=(\alcp{i,1},\ldots,\alcp{i,m_i})$, $1\le i\le n$, 
such that $m_n>1$ and $\alcp{n,m_n}=\mu_\tau$,  
where $\tau=\taucp{n,m_n-1}$, $\tauvec$ denoting the chain associated with $\alvec$. 
Then $\alvec$ is called a \emph{principal chain (to base $\tau$)},\index{principal chain (to base $\tau$)} 
and $\tau$ is called \emph{the base of $\alvec$}.\index{base of $\alvec$}
If $\alvec\in M$ then we say that $\alvec$ is a \emph{principal chain in $M$} and that $\tau$ is a \emph{base in $M$}.
\item\label{parpart} Let $\alvec$ be as in part \ref{principalchainpart} and $\bevec\in M_\alvec$,
where $\bevec_i=(\becp{i,1},\ldots,\becp{i,k_i})$, $1\le i\le l$, with associated chain $\tauvec$. 
If $\alvec\sub\bevec$ let $r\in(n,l]$ be minimal such that $\tau\nmid\becp{r,1}$, i.e.\ $\taucp{r,1}<\tau$, 
and $\taurstar<\taucp{r,1}$, if that exists. Otherwise set $r:=0$.
Then 
\[\parind_{M_\alvec}(\bevec):=r\] 
is called the \emph{parameter index of $\be$ in $M_\alvec$}, \index{parameter index, $\parind$}
and the \emph{parameter of $\bevec$ in $M_\alvec$}\index{parameter of ($\bevec$ in) $M_\alvec$, $\param(M_\alvec)$} 
is defined by 
\[\param_{M_\alvec}(\bevec):=\left\{\begin{array}{ll}
\taucp{r,1} & \mbox{ if }r>0\\  
0 & \mbox{ otherwise.}
\end{array}\right.\] 
We will omit the subscript $M_\alvec$ when this 
context is unambiguous.
The set of parameters of $M_\alvec$ is then defined by
\[\param(M_\alvec):=\{\param(\bevec)\mid\bevec\in M_\alvec\},\]
and its maximum is denoted by $\maxparam(M_\alvec):=\max(\param(M_\alvec))$. \index{$\maxparam$} 
\item\label{mdbpart} Suppose in addition to the assumptions of part \ref{parpart} that $M_\alvec$ is
spanning above $\alvec$ and that $\bevec$ is the tracking chain of $\max(\ov[M_\alvec])$, i.e.\ $\bevec$ is the $\ktc$-maximum of $M_\alvec$.
In the case where
\[\max\{\param(\gavec)\mid\gavec\in M_\alvec\andsp\gavec\nsubseteq\bevec\}<\param(\bevec)\in\Ez\]
and either $\alvec$ is convex or $\lh(\ov(\alvec))=\ov(\bevec)$, we call
\[\db(M_\alvec):=\param(\bevec)\] 
the \emph{distinguished base of $M_\alvec$}\index{distinguished base of $M_\alvec$, $\db$} and 
\[\dc(M_\alvec):=\bevec_{\restriction_{r,1}}\mbox{, where }r=\parind(\bevec),\] 
the \emph{distinguished chain in $M_\alvec$}.\index{distinguished chain in $M_\alvec$, $\dc$} 
In all other cases we set $\db(M_\alvec):=0$ and $\dc(M_\alvec):=()$.
\item\label{mdspart} Let $M_\alvec$ and $\bevec$ be as in part \ref{mdbpart} and suppose that $\si:=\db(M_\alvec)>0$.
If $(n,m_n+1)\in\dom(\tauvec)$, i.e.\ $m_n<k_n$, define $\sivec_0:=(\taucp{n,m_n},\ldots,\taucp{n,k_n-1})$,
otherwise set $\sivec_0:=()$. Then define $\sivec_j:=(\taucp{n+j,1},\ldots,\taucp{n+j,k_{n+j}-1})$ for $j=1,\ldots,r-n-1$,
and $\sivec_{r-n}:=(\si)$. 
Finally define $\sivec$ to be the concatenation of the vectors $\sivec_j$, $j=0,\ldots,r-n$.  
The \emph{distinguished sequence of $M_\alvec$}\index{distinguished sequence of $M_\alvec$, $\ds$} is then defined by 
\[\ds(M_\alvec):=\sivec,\] 
and in all cases where the above conditions are not met we set $\ds(M_\alvec):=()$. 
\end{enumerate}
\end{defi}

\begin{lem}\label{dslem}
For convex $\alvec$ let $M_\alvec$ and $\bevec$ be as in part \ref{mdspart} of Definition \ref{Mdbdefi}. 
Setting $\sivec=(\si_1,\ldots,\si_s):=\ds(M_\alvec)$ and $\xi:=\gs(M_\alvec)$ we have \[\sumend(\xi)=\si_1.\]
Taking for $\alvec$ in the context of Case 2 of Definition \ref{sivecdefi} the same $\alvec$ as here if $(n,m_n+1)\in\dom(\tauvec)$ and 
$\alvec^\frown(\xi)$ otherwise, the sequence $\mq^{(\sipr)}(\si)$ of Lemma \ref{simqchilem} coincides with the distinguished sequence 
$\sivec$, where $\sipr$ is equal to the base $\tau$ of $\alvec$ if $\xi\ge\tau$.
\end{lem}
\begin{proof} This is a consequence of Definition \ref{spanningdefi} and Remark \ref{intermedmurmk} regarding intermediate $\nu$-indices 
which are obtained by application of the $\la$-operator. \qed \end{proof}

\noindent Recall the operator $\bardot$ from Section 8 of \cite{W07a} and Section 5 of \cite{CWa}. 
\begin{defi}[Closedness]\label{cldefi}\index{closed, closed above $\alvec$}
Let $M\finsub\TC$ be spanning (spanning above $\alvec$). $M$ is \emph{closed} (\emph{closed above $\alvec$}) if and only if
for all principal chains $\bevec$ in $M$ such that $\taubar\in(\taupr,\tau)$,
where $\tau$ is the base of $\bevec$ and $\tauvec$ denotes the chain associated with $\bevec$, 
we have 
\begin{equation}\label{supportcond}
   \maxparam(M_\bevec)\left\{\begin{array}{ll}
            \ge\taubar&\mbox{if } \db(M_\bevec)=0\\[2mm]
            >\taubar&\mbox{otherwise}.
            \end{array}\right.
\end{equation}
We call the base $\tau$ of a principal chain $\bevec$ in $M$ such that $\taubar\in(\taupr,\tau)$ a 
\emph{supported base in $M$}\index{supported base in $M$}
if and only if (\ref{supportcond}) holds, otherwise we call $\tau$ a \emph{non-supported base in $M$}.
\end{defi}

\begin{lem}\label{closedbarlem}
Let $M$ be closed (closed above $\alvec$). Then $M$ is closed under $\bardot$ (above $\alvec$) in the following sense:
for any principal chain $\bevec\in M$ with supported base $\tau=\taucp{i,j}$ as in the above definition there exists 
a principal chain $\gavec\in M_\bevec$, $\bevec\sub\gavec$, to base $\taubar=\taucp{r,s}$ such that the bases $\taupr=\taucppr{i,j}$ and 
$\taubarpr=\taucppr{r,s}$ have the same index pair $(i,j)^\prime=(r,s)^\prime$.  
\end{lem}
\begin{proof}
The claim follows from closedness by induction on the height of $M_\bevec$, using Lemma 5.9 of \cite{CWa}.
\qed \end{proof}

Recall the notation $\al^\Ez$ introduced in Section 2 of \cite{W07a} for the least epsilon number strictly greater than $\al$.
The following lemma provides a crucial estimation of the term parameters, see Definition 3.28 of \cite{W07a}, 
in closed sets of tracking chains.

\begin{lem}\label{closedparamlem}
Let $\bevec$ be a convex principal chain to base $\tau$, with associated chain $\tauvec$, and let $M=M_\bevec$ be closed above $\bevec$.
Then for all $\gavec\in M$ and all $(i,j)\in\dom(\gavec)-\dom(\bevec)$ such that either $r:=\parind_{M}(\gavec)=0$ or
$(i,j)\kglex(r,1)$ we have \[\Part(\gacp{i,j})\subseteq\maxparam(M)^\Ez.\] 
\end{lem}
\begin{proof} In the notation of part \ref{principalchainpart} of Definition \ref{Mdbdefi} we have $\tau=\taucp{l,k_l-1}$; thus the setting of 
relativization of $M$ is given by $\sivec_0:=\cs(\bevec_{\restriction_{l,k_l-1}})\in\RS$. The indices $\gacp{i,j}$ can therefore
be considered as elements of $\Tsivec$, where $\sivec_0\subseteq\sivec$ is according to nestings of $\iota$-operators involved in
the application of $\mu$- and $\la$-operators, see Definition \ref{Ttauvec}.
The lemma now follows by induction on $\lsivec(\gacp{i,j})$ for the appropriate extension $\sivec$ of $\sivec_0$, since for epsilon numbers
$\ga$ the parameters $\Part(\ga)$ are contained in $\Part(\gabar)$ and (possibly modulo $\omega$-exponentiation) $\Part(\laga)$, 
and for ordinals of a form $\ga=\tht^\si(\eta)>\tau$ where $\eta<\Om_1$, 
i.e.\ which are not epsilon numbers, we have $\Part(\ga)=\Part(\eta)$, cf.\ equation \ref{logred}.
\qed \end{proof}

\begin{defi}[Closure]\label{closuredefi}\index{closed!closure}
Let $\alvec\in\TC\cup\{()\}$ and $M=M_\alvec\finsub\TC$ be a set of tracking chains as in Definition \ref{Malgsdefi}.
We define the \emph{closure of $M$ above $\alvec$}, denoted as $\Mcl$, to be the least set of tracking chains that contains $M$ and is closed under
clauses {1 -- 8}, relaxed by the condition that in the case $\alvec\not=()$ the respective resulting tracking chains $\bevec$ 
satisfy $\ov(\alvec)\lo\ov(\bevec)$, cf.\ Definitions \ref{precldefi} and \ref{spanningdefi} and Definition \ref{spanningabovedefi}, 
and that
\begin{enumerate}
\item[9.] \emph{supports bases:} if $\bevec$ is a principal chain in $M$ to base $\tau$ such that $\taubar\in(\taupr,\tau)$
then $\bevec^\frown(\taubar)\in M$, unless condition \ref{supportcond} of Definition \ref{cldefi} holds anyway.
\end{enumerate} 
\end{defi}
\begin{rmk} Notice that the above clause for base support makes a choice in the support of bases. The process of closure
is finite since application of the operator $\bardot$ strictly lowers the $\operatorname{l}$-measure, see inequality \ref{barlen}
following Definition \ref{Ttauvec}. 
\end{rmk}
\begin{defi}[Essential closedness]\label{essclsddefi}\index{closed!essentially closed} Let $M\finsub\TC$ be spanning (spanning above $\alvec$). 
M is \emph{essentially closed (above $\alvec$)} 
if and only if the closure $\bar{M}$ of $M$ under initial chains ($\bevec$ such that $\ov(\alvec)\lo\ov(\bevec)$) 
is closed (closed above $\alvec$)
and only adds tracking chains of a form $\bevec^\frown(\ga_1)$ where $\bevec^\frown(\ga_1,\ga_2)\in M$ for some $\ga_2$ such as 
$\mu_{\sumend(\ga_1)}$.
\end{defi}
\begin{rmk} Essentially closed sets remain to be closed under $\me$, $\lh$, and $\lh_2$ in the sense of Corollaries 
\ref{lhclscor} and \ref{lhtwoclscor}.
\end{rmk}

The following definition of essential closure of a given set $M$ of tracking chains allows us to omit redundant chains.
Such chains do not belong to the original set $M$, end in a $\ka$-index, and have 1-step extensions in $\Mcl$, but only by $\nu$-indices. 

\begin{defi}[Essential closure]\label{esscldefi}\index{closed!essential closure}
Let $\alvec\in\TC\cup\{()\}$ and $M=M_\alvec\finsub\TC$ be a set of tracking chains as in Definition \ref{Malgsdefi}. 
The \emph{essential closure of $M$ above $\alvec$}, denoted as $\Mecl$, is obtained from $\Mcl$ by dropping all tracking chains 
$\gavec\in\Mcl-M$ that are of a form $\bevec^\frown(\ga_1)$ where $\bevec^\frown(\ga_1,\ga_2)\in \Mcl$ for some $\ga_2$ and
for which there does not exist any proper extension of a form $\gavec^\frown\gavecpr\in\Mcl$.
\end{defi}

We are now prepared to introduce the notions of $\ka$-index and base minimization. These provide the key tools in the algorithm
given by Theorem \ref{isomtheo} that assigns isominimal realizations to given respecting forests of order $2$ by determining 
minimal (relativized) $\leo$- and $\letwo$-components, respectively, that satisfy a given forest.

\begin{defi}[$\kavec$-index minimization]\label{kappaindexdefi}\index{$\kavec$-index minimization} 
Let $\alvec$ be either the empty sequence or a convex tracking chain, where
$\alvec_i=(\alcp{i,1},\ldots,\alcp{i,m_i})$, $1\le i\le n$, $n\ge 0$. Let $M=M_\alvec\finsub\TC$ be a set of proper extensions of $\alvec$ 
of the form $M=\{\bevec\}\cup M_\bevec$ for a convex tracking chain $\bevec$ with associated chain $\tauvec$
such that $M_\bevec$ is either empty or closed above $\bevec$ and either
\begin{enumerate}
\item $\bevec=\alvec^\frown(\becp{n+1,1})$, where we set $\tau:=\taucp{n+1,1}$, or
\item $\bevec=\alvec^\frown(\becp{n+1,1},\mu_\tau)$ where $\tau:=\taucp{n+1,1}$ or 
\item $\bevec$ extends $\alvec$ by the $\nu$-index $\becp{n,m_n+1}=\mu_\tau$ where $\tau:=\taucp{n,m_n}$, $n>0$. 
\end{enumerate}
Set $\xi:=\gs(M_\bevec)$ and suppose $\si$ to be either the base of a $\ktwo$-predecessor $\ga$ of $\be:=\ov(\bevec)$, $\gavec:=\tc(\ga)$, 
or $\si=1$ and $\ga=0$, $\gavec:=()$,
such that all $\ktwo$-predecessors $\de\le\be$ of ordinals in $\ov[M_\bevec]$ satisfy $\de\le\ga$.
We call $\gavec$ the \emph{chain of the preserved $\ktwo$-predecessor} and $\si$ its \emph{base}.
Note that $\si$ and $\gavec$ determine each other and that 
according to the assumptions $\be$ does not have any $\ktwo$-successor in $\ov[M_\bevec]$.
Setting $\eta:=0$ in the case $\alvec=()\andsp M_\bevec=\emptyset$, and $\eta:=\si\cdot\om^\xi$ otherwise,
if $\bevec$ is of the form 1 we then have $\si\mid\becp{n+1,1}$ and, moreover, $\eta\le\becp{n+1,1}<\rho_n$ since $\si\le\taunestar$, and
if $\bevec$ is of the form 2 or 3 we have $\xi<\tau$ and hence $\eta<\tau<\rho_n$.
\\[2mm]
We define the \emph{$\ka$-index minimization above $\si$ in $M$ at $\bevec$}, denoted as $\bfkaarg{M,\bevec,\si}$, or equivalently
the \emph{$\ka$-index minimization in $M$ at $\bevec$ preserving $\gavec$}, denoted as $\bfkaarg{M,\bevec,\gavec}$, and $\kavec$ in short, as follows.
\[\kavec(\bevec):=\alvec^\frown(\eta)\quad\mbox{ and }\quad\kavecind:=\eta,\]
and for $\devec\in M_\bevec$ we define $\kavec(\devec)$ by considering the following cases.\index{$\kavec(\alvec)$, $\kavecind$}
\\[2mm]
{\bf Case 1:} $\bevec=\alvec^\frown(\becp{n+1,1})$ and $\xi=\taucp{n+1,1}\in\Ez^{>\taunestar}$.
Then we only change the index $\becp{n+1,1}$ at $(n+1,1)$ in $\devec$ to $\taucp{n+1,1}$ in order to obtain $\kavec(\devec)$,
which we call a \emph{horizontal translation}.\index{horizontal translation} 
\\[2mm]
{\bf Case 2:} Otherwise. Then we have $\xi<\tau$.
\\[2mm]
{\bf Subcase 2.1:} $\xi=\eta$ and $\bevec^\frown(\xi)\sub\devec$.
Then we define for $\devec=\bevec^\frown(\xi)^\frown\devecpr$
\[\kavec(\devec):=\alvec^\frown(\eta,1)^\frown\devecpr,\] and for $\devec$ of a form 
$\bevec^\frown(\xi,\decp{n+2,2},\dots,\decp{n+2,k_{n+2}})^\frown\devecpr$ we define 
\[\kavec(\devec):=\alvec^\frown(\eta,1+\decp{n+2,2},\decp{n+2,3},\ldots,\decp{n+2,k_{n+2}})^\frown\devecpr.\]
{\bf Subcase 2.2:} Otherwise. Then we simply replace the initial sequence $\bevec$ of $\devec$ by 
$\alvec^\frown(\eta)$ in order to obtain $\kavec(\devec)$, i.e., writing $\devec$ in the form $\bevec^\frown\devecpr$ we 
define \[\kavec(\devec):=\alvec^\frown(\eta)^\frown\devecpr.\]
\end{defi}

\begin{theo}\label{kappaindextheo} Let $M,\alvec,\bevec$ and $\si,\gavec$ be as in the above definition as well as the shortcuts 
$\be,\ga$, and set $\al:=\ov(\alvec)$, $\be_\ka:=\ov(\kavec(\bevec))$, and $I:=I(\gavec)$. 
Then $\kavec[M]$ is a set of tracking chains and we have 
\begin{enumerate}
\item $\al\lo\ov\circ\kavec[M]$, 
\item $\predecs_2(\be_\ka)=\{\de\mid\de\letwo\ga\}$, if $\ga\not=0$, otherwise $\be_\ka$ is $\letwo$-minimal.
\item the images of $I\cup M$ and $I\cup \kavec[M]$ under $\ov$ are isomorphic substructures of $\Ctwo$, and 
\item $\kavec[M_\bevec]$ is closed above $\kavec(\bevec)$.
\item $\kavec[M]$ is closed above $\alvec$.
\end{enumerate}
\end{theo}
\begin{proof} The theorem directly follows from the definitions involved.
\qed \end{proof}

We turn to base minimization in sets of tracking chains. This provides a tool to determine $\lepw$-minimal isomorphic copies
of sets of tracking chains.
Recall the notion of base transformation, see Section 5 of \cite{W07a} or in short Definition 2.15 of \cite{W07c}.
For convenience we set $\piarg{\tau,\tau}:=\id$.

\begin{defi}[Base minimization]\label{pitcdefi}\index{base minimization} 
Let $\alvec$ be either the empty sequence or a convex tracking chain, where
$\alvec_i=(\alcp{i,1},\ldots,\alcp{i,m_i})$, $1\le i\le n$, $n\ge 0$. Let $M=M_\alvec\finsub\TC$ be a set of proper extensions of $\alvec$ of 
a form $M=\{\bevec\}\cup M_\bevec$, where $\ov[M_\bevec]$ contains a $\ktwo$-successor of $\ov(\bevec)$, 
$M_\bevec$ is closed above $\bevec$, and $\bevec$ is a convex principal chain in $M$ to base $\tau$, consisting of the vectors
$\bevec_i=(\becp{i,1},\dots,\becp{i,k_i})$, $1\le i\le l$, such that either 
\begin{enumerate}
\item $\bevec=\alvec^\frown(\becp{n+1,1},\mu_\tau)$ with $\becp{n+1,1}=_\NF\bepr+\tau$ or 
\item $\bevec$ extends $\alvec\not=()$ by the $\nu$-index $\becp{n,m_n+1}=\mu_\tau$, $\bepr:=0$. 
\end{enumerate}
Set $\al:=\ov(\alvec)$, $\be:=\ov(\bevec)$, $\xi:=\gs(M_\bevec)\ge\tau$, $\si_0:=\db(M_\bevec)$, and $\si_1:=\max\{\maxparam(M_\bevec),\rho\}$, 
where $\rho$ is either the base of a $\ktwo$-predecessor $\de$ of $\be$,
setting $\devec:=\tc(\de)$, or $\rho=1$, setting $\de:=0$ and $\devec:=()$,
such that all greatest $\ktwo$-predecessors $\ga<\be$ of ordinals in $\ov[M_\bevec]$ satisfy $\ga\le\de$. 
We call $\devec$ the \emph{chain of the preserved $\ktwo$-predecessor} and $\rho$ its \emph{base}.
Note that $\rho$ and $\devec$ determine each other. 
\\[2mm]
We define the \emph{base minimization above $\rho$ in $M$ at $\bevec$}, $\bfpiarg{M,\bevec,\rho}$, or equivalently the
\emph{base minimization in $M$ at $\bevec$ preserving $\devec$}, $\bfpiarg{M,\bevec,\devec}$, as follows,
where we simply write $\pivec$, whenever the arguments $M,\bevec$, and $\rho$ or $\devec$ are understood from the context.
In order to define $\pivec(\gavec)$ for $\gavec\in \{\bevec\}\cup M_\bevec$ we consider the following cases. 
\\[2mm]
{\bf Case 1:} $\si_0\le\rho$ or otherwise $\piarginv{\si_0,\tau}(\la_{\si_0})<\xi$.
Let $\si\in\Ez\cap(\si_1,\tau]$ be minimal such that $\xi\le\pistinv(\la_\si)$.
Minimality of $\si$ then implies that $\xi=\pistinv(\la_\si)$, see Lemmata 5.8 and 8.2 of \cite{W07a}.
In the case $\si=\tau$ transformation to a smaller base is not possible, and if assumption 2 holds for $\bevec$ then we set $\pivec:=\id$. 
Otherwise define 
\[\pivec(\bevec):=\alvec^\frown(\si,\mu_\si),\] 
and for $\gavec\in M_\bevec$ and $r:=\parind_{M_\bevec}(\gavec)$
we either have $r=0$ or $r>l$ and define
\begin{equation}\label{pitceq}\pivec(\gavec)_{i,j}:=\left\{\begin{array}{ll}
            \pivec(\bevec)_{i,j}&\mbox{if } (i,j)\in\dom(\bevec)\\[2mm]
            \gacp{i,j}&\mbox{if } r>0\andsp(r,1)\klex(i,j)\\[2mm]
            \pist(\gacp{i,j})&\mbox{otherwise,}
            \end{array}\right.
\end{equation}
which in the case $\si=\tau\andsp\tau<\becp{n+1,1}$ performs a horizontal translation, cf.\ Definition \ref{kappaindexdefi}.
\\[2mm]
{\bf Case 2:} $\si:=\si_0>\rho$ and $\xi\le\pistinv(\la_\si)$.
\\[2mm]
{\bf Subcase 2.1:} $\tau\nmid\xi$. Then, due to the uniqueness of $\si=\db(M_\bevec)$, we have $\xi=\tau\cdot\nu+\si\in[\tau,\lat]$ 
for some $\nu>0$, which we write as $\nu=\la+k\minusp\chit(\la)$, where $\la\in\Lim\cup\{0\}$ and $k<\om$ such that
if $\chit(\la)=1\andsp\nu=\la$ then $k=1$.
We set $\eta:=\om^{\pist(\la)+k}$ and define
\[ \pivec(\bevec):=\alvec^\frown(\si,\eta).\] 
For $\gavec\in M_\bevec$ such that $\gavec\ktc\bevec^\frown(\xi)$ we define $\pivec(\gavec)$ as in (\ref{pitceq}) of Case 1.
For $\gavec\in M_\bevec$ such that $\bevec^\frown(\xi)\letc\gavec$, which we may write as 
$\gavec=\bevec^\frown(\xi,\gacp{l+1,2},\ldots,\gacp{l+1,k_{l+1}})^\frown\gavecpr$, we define
\[\pivec(\gavec):=\left\{\begin{array}{ll}
            \alvec^\frown(\si,\eta+1)^\frown\gavecpr&\mbox{if } k_{l+1}=1 \\[2mm]
            \alvec^\frown(\si,\eta+1+\gacp{l+1,2},\gacp{l+1,3},\ldots,\gacp{l+1,k_{l+1}})^\frown\gavecpr&\mbox{otherwise.}
            \end{array}\right.
\]
{\bf Subcase 2.2:} $\tau\mid\xi$. We then have $\xi=\tau\cdot\nu$ for some $\nu>0$ which we write as $\nu=\la+k\minusp\chit(\la)$
where $\la\in\Lim$ and $k<\om$. According to the definition of $\si$ we have $k\minusp\chit(\la)=0$.
Lemma \ref{dslem} shows that $\chit(\la)=0$ and hence $k=0$. According to our assumptions $\si$ has a 
unique occurrence in $\la$ and $\max(\Part(\la))=\si$, and we may apply $\pist$ to $\la$, simply leaving $\si$ unchanged, 
thus obtaining $\chis(\pist(\la))=1$. We set $\eta:=\om^{\pist(\la)}$, so that $\chis(\eta)=1$, and define
\[ \pivec(\bevec):=\alvec^\frown(\si,\eta).\] 
For $\gavec\in M_\bevec$ such that $\gavec\ktc\dc(M_\bevec)$ we define $\pivec(\gavec)$ again as in (\ref{pitceq}) of Case 1.
For $\gavec\in M_\bevec$ such that $\dc(M_\bevec)\letc\gavec$, which, setting 
$r:=\parind_{M_\bevec}(\max(M_\bevec))=\parind_{M_\bevec}(\gavec)$, we may write as 
$\gavec={\gavec_{\restriction_{r-1}}}^\frown(\gacp{r,1},\ldots,\gacp{r,k_r})^\frown\gavecpr$, we define
\[\pivec(\gavec):=\left\{\begin{array}{ll}
            \alvec^\frown(\si,\eta+1)^\frown\gavecpr&\mbox{if } k_r=1 \\[2mm]
            \alvec^\frown(\si,\eta+1+\gacp{r,2},\gacp{r,3},\ldots,\gacp{r,k_r})^\frown\gavecpr&\mbox{otherwise.}
            \end{array}\right.
\]
This concludes the definition of $\pivec=\bfpiarg{M,\bevec,\rho}=\bfpiarg{M,\bevec,\devec}$, and for convenience we introduce the notations 
\index{$\pivec(\alvec)$, $\pivecind$}
\[\pivecind:=\si\] 
and
\[\alvec^+_\pi:=\alvec^\frown(\si),\]
unless we have $\si=\tau$ in assumption 2 for $\bevec$, where we set $\alvec^+_\pi:=\pivec(\bevec)$.
\end{defi}

\begin{theo}\label{pitctheo} Let $M,\alvec,\bevec$, and $\rho,\devec$ be as in the above definition as well as the shortcuts
$\al,\be,\de,\si,\tau$, and set $I:=I(\devec)$ and $\be_\pi:=\ov(\pivec(\bevec))$. 
Then $\pivec[M]$ is a set of tracking chains, and we have 
\begin{enumerate}
\item $\al\lo\ov\circ\pivec[M]$,
\item $\predecs_2(\be_\pi)=\{\ga\mid\ga\letwo\de\}$ if $\de\not=0$, 
otherwise $\be_\pi$ is $\letwo$-minimal,
\item the images of $I\cup M$ and $I\cup \pivec[M]$ under $\ov$ are isomorphic substructures of $\Ctwo$, 
\item $\pivec[M_\bevec]$ is closed above $\pivec(\bevec)$, and
\item $\{\alvec^+_\pi\}\cup\pivec[M]$ is closed above $\alvec$, hence $\pivec[M]$ is essentially closed above $\alvec$.
\end{enumerate}
\end{theo}
\begin{proof} Due to Lemma \ref{closedparamlem}, all terms to which the order preserving base transformation $\pist$ is applied, 
use parameters below $\si$ 
(with the unique exception handled explicitly in Subcase 2.2) and can be translated into $\Tt$, see Section 6 of \cite{W07a}, 
invariantly regarding localization (Lemma 6.5 of \cite{W07a}), the operator $\bardot$ and fine-localization (Lemma 5.7 of \cite{CWa}), 
the operators $\ze,\la,\mu$ (Lemmata 6.8 and 7.7 of \cite{W07a} and Lemma 3.6 of \cite{CWc}), hence also regarding tracking sequences. 
We have verified commutativity of $\pist$ with $\ze,\la,\mu$ (Lemmata 5.6, 7.10 of \cite{W07a} and Lemma 3.7 of \cite{CWc}), 
with $\bardot$ (Lemma 5.7 of \cite{CWa}), and also with the indicator $\chi$ and the operator $\varrho$ (Lemmata 3.2 and 3.11 of \cite{CWc}). 
For $\chi$ and $\varrho$, however, we need full commutativity with $\pist$ with respect to
the base argument as well, i.e. \[\chiga(\eta)=\chi^{\pist(\ga)}(\pist(\eta)) \quad\mbox{ and }\quad 
\pist(\varrho^\ga_\eta)=\varrho^{\pist(\ga)}_{\pist(\eta)}\]
for suitable arguments $\ga$ and $\eta$.  
For $\chi$ this property obviously holds; hence it also follows for $\varrho$.
Inspecting the translation mapping we also observe that
\[\pist(\tht^\ga(\eta))=\tht^{\pist(\ga)}(\pist(\eta))\]
for suitable arguments $\ga$ and $\eta$. Commutativity of $\pist$ with addition, multiplication, $\om$-exponentiation, and $\log$ is obvious.
Therefore $\pist$ also commutes with maximal (1-step) extensions ($\me$), see Definition 5.2 of \cite{CWc}. 

In Case 2.1 we have $\rhoargs{\si}{\eta}+\si\le\lasi$ as a consequence of Case 2 and continuity in $\si$, since due to 
Lemma \ref{closedparamlem} $\rhoargs{\si}{\eta}=\si\cdot(\pist(\la)+k\minusp\chis(\pist(\la)))=\pist(\tau\cdot\nu)$. 
The weak monotonicity of $\rhoargs{\si}{}$ now implies that $\eta<\mu_\si$, since otherwise 
$\rhoargs{\si}{\eta}+\si\ge\rhoargs{\si}{\mu_\si}+\si>\lasi$ using part a) of Lemma 3.12 of \cite{CWc}.
 
In Case 2.2 we have $\rhoargs{\si}{\eta}=\si\cdot\pist(\la)=\pist(\xi)\le\lasi$ as a consequence of Case 2 and continuity with
respect to the single occurrence of the parameter $\si$ in $\la$ and hence in $\xi$. This entails
$\rhoargs{\si}{\eta}+\zeta_\si\le\lasi$ as $\ze_\si<\si$.
We thus obtain $\eta<\mu_\si$, since equality is ruled out by $\chis(\eta)=1$, and the assumption $\eta>\mu_\si$ would
imply $\rhoargs{\si}{\eta}+\zeta_\si>\rhoargs{\si}{\mu_\si}+\zeta_\si=\lasi$.

Close inspection of the definition of $\pivec$ now shows that $\pivec[M]\subseteq\TC$ and that $I\cup \pivec[M]$ is isomorphic to $I\cup M$. 
Finally, closedness of $\pivec[M_\bevec]$ above $\pivec(\bevec)$ is seen by inspection of Definitions 
\ref{precldefi}, \ref{spanningdefi}, Definition \ref{cldefi}, and closedness of $\{\alvec^+_\pi\}\cup\pivec[M]$ above $\alvec$
follows from the choice of $\pivec(\bevec)$.
\qed \end{proof}

\section{Isominimal realization}\index{isominimal realization}

\begin{defi}\label{alcovdefi}
Let $P\not=\emptyset$ be finite such that $P_a:=\{b_1,\ldots,b_r,a\}\stackrel{\cdot}{\cup} P$\index{$P_a$} is a respecting forest of order $2$
over the language $(0;\le,\leo,\letwo)$, where the constant $0$ does not need to be interpreted,
$r\ge 0$,
\[b_1\ktwo\ldots\ktwo b_r\ktwo a < P, \quad\mbox{and}\]
\[a\lo\max(P) \quad\mbox{ if }a>0.\]
Suppose that $\al=0$ if $a=0$ and otherwise $\al<1^\infty$ such that $\predecs_2(\al)=\{\be_1,\ldots,\be_r\}$ for ordinals $\be_1<\ldots<\be_r$.
\begin{enumerate}
\item A mapping $c_\al:P_a\to \Ctwo$ is called an \emph{$\al$-covering of $P_a$}\index{$\al$-covering of $P_a$, $c_\al$} if and only if
$c_\al(b_i)=\be_i$ for $i=1,\ldots,r$, $c_\al(a)=\al$, and $\Image(c_\al)$ is a cover of $P_a$ in $\Ctwo$.
\item An $\al$-covering $c_\al$ of $P_a$ is called an \emph{$\al$-isomorphism of $P_a$}\index{$\al$-isomorphism of $P_a$} 
if $\Image(c_\al)$ is isomorphic to $P_a$. 
\item $c_\al$ is called an \emph{isominimal realization of $P_a$ above $\al$}\index{isominimal realization!isominimal realization of $P_a$ above $\al$} if and only if it is an $\al$-covering that is $\lepw$-minimal
among all $\al$-coverings of $P_a$.
\item An $\al$-covering $c_\al$ is called \emph{convex}\index{convex!convex $\al$-covering} if $\tc(\be)$ is convex for all $\be\in\Image(c_\al)$.
\end{enumerate}

\noindent Let $\alvec\in\TC\cup\{()\}$ and $M_\alvec$ be (essentially) closed above $\alvec$.
Setting $\al:=\ov(\alvec)$, $\predecs_2(\al)=:\{\be_1,\ldots,\be_r\}$, we define the \emph{(respecting) forest associated with $M_\alvec$}
\index{respecting forest!(respecting) forest associated with $M_\alvec$}
to be $P_a$, where $P:=\ov[M_\alvec]$, $a:=\al$, $b_i:=\be_i$  for $i=1,\ldots,r$, and $P_a:=\{b_1,\ldots,b_r,a\}\cup P$.
\end{defi}
\begin{rmk} For any respecting forest $P_a$ of order $2$, as in the above definition, there exists a convex $\al$-covering: 
we may simply choose the proof theoretic ordinal of a theory $\mathrm{ID}_N$ for a suitable
index $N<\om$ (setting $\mathrm{ID}_0:=\mathrm{PA}$), which provides a sufficiently long $\ktwo$-chain to cover $P_a$.
\end{rmk}
\begin{theo}\label{isomtheo}
Let $P_a$ be a respecting forest of order $2$ as in the above definition, with a given convex $\al$-covering $c_\al$, and set
$\alvec:=\tc(\al)$ if $\al>0$, and $\alvec:=()$ if $\al=0$. 
There exists 
a unique $\al$-isomorphism $i_\al$ of $P_a$ such that 
\begin{enumerate}
\item $i_\al[P]$ is closed under $\lh,\lh_2$ and 
\item $\tc\circ i_\al[P]$ is essentially closed above $\alvec$.
\end{enumerate}
\end{theo}
\begin{proof} We argue by induction on the cardinality of $P$. 
Note that property 1 follows from property 2 by Corollaries \ref{lhclscor}, \ref{lhtwoclscor}, and Lemma \ref{strongspanlem}.
Let $\alvec_i=(\alcp{i,1},\ldots,\alcp{i,m_i})$, $1\le i\le n$, $n\ge0$, be the components of $\alvec$. 
Let $P=\bigcup_{i=1}^{k+1} P_i$ be the partitioning of $P$ into increasing $a$-$\leo$-connectivity components.
Let $Q$ be any of the $P_i$ and set $q:=\min(Q)$, i.e., in the case $a=0$ the element $q$ is the $i$-th $\leo$-minimal element in $P$,
and otherwise $q$ is the $i$-th immediate $\lo$-successor of $a$ in $P$. 
Then the restriction of $c_\al$ to $Q_a$ remains to be a convex $\al$-covering, 
and we may assume that $\be:=c_\al(q)$ does not have any $\ktwo$-predecessor in 
$(\al,\be)$, since otherwise we would obtain another convex $\al$-covering of $P_a$ by simply replacing $\be$ by such a $\ktwo$-predecessor.
The convexity of $\alvec$ furthermore implies that $\alvec\subseteq\bevec:=\tc(\be)$ where $\bevec_i=(\becp{i,1},\ldots,\becp{i,k_i})$,
$1\le i\le l$, and the convexity of $c_\al$ implies that $\bevec$ is convex. 
Let $\tauvec$ be the associated chain. In the case where $k_l=1$ we have $l>n$ and $l^\star\klex(n,m_n)$ due to the 
$\al$-$\letwo$-minimality of $\be$, and $q$ does not have any $\ktwo$-successor in $P$. If $k_l>1$ then $\bevec$ is a principal chain, 
and due to the $\al$-$\letwo$-minimality of $\be$ we either have $(l,k_l)=(n,m_n+1)$ and $\becp{l,k_l}=\mu_\taucp{n,m_n}$, or $l>n$, 
$k_l=2$, $\becp{l,2}=\mu_\taucp{l,1}$, and $l^\star\klex(n,m_n)$. 
In the cases where $l>n$ we may assume that $l=n+1$ due to Theorem \ref{verttrantheo}.  

Now, if necessary, the i.h.\ is applied to $P_q$, defined as the substructure of $P_a$ given by the union of the subset of 
$\{b_1,\ldots,b_r,a,q\}$ matching the $\letwo$-predecessors 
of $\be$ with the set of elements of $P$ that are $\lo$-successors of $q$, and the appropriate restriction of $c_\al$. 
We thus obtain (in the non-trivial case) a $\be$-isomorphism $i_\be$ and define $M_\bevec$ to be the closure of 
$\tc\circ i_\be[Q^{>q}]$ under initial chains,
so that $M_\bevec$ is either empty or closed above $\bevec$, cf.\ Definition \ref{essclsddefi}. 
Setting for convenience $b_{r+1}:=a$ and $\be_{r+1}:=\al$, let $\si$ be the base of $\be_i$ where $b_i$ is the greatest 
$\ktwo$-predecessor of $q$ in $P_a$ if such exists and $\si:=1$ otherwise. We now define the set $M:=\{\bevec\}\cup M_\bevec$ of proper
extensions of $\alvec$ and consider the following two cases.
\\[2mm]
{\bf Case 1:} $q$ does not have any $\ktwo$-successor in $P$. Here we may apply $\ka$-index minimization above $\si$ in $M$ at $\bevec$, 
see Definition \ref{kappaindexdefi} and Theorem \ref{kappaindextheo}, and set $M_q:=\kavec[M]$, $\bevec_q:=\kavec(\bevec)$,
$\be_q:=\ov(\bevec_q)$, and $\xi_q:=\kavecind$. 
\\[2mm]
{\bf Case 2:} Otherwise, base minimization above $\si$ in $M$ at $\bevec$ applies, 
see Definition \ref{pitcdefi} and Theorem \ref{pitctheo}, and we set $M_q:=\{\alvec^+_\pi\}\cup\pivec[M]$, $\bevec_q:=\pivec(\bevec)$,
$\be_q:=\ov(\bevec_q)$, and $\xi_q:=\pivecind$.
\\[2mm]
Now $M_q$ is closed above $\alvec$, and using straightforward translation we can define the mapping $i_\al$ on $P$. 
We have $\ka$-indices $\xi_{q_i}$ for $i=1,\ldots,k+1$ where $q_i=\min(P_i)$ and define $\xi_i:=\sum_{j=1}^i\xi_{q_j}$ for $i=1,\ldots,k+1$. 
Changing the $\ka$-index $\xi_{q_i}$ to $\xi_i$ at $(n+1,1)$ in every chain in $M_{q_i}$ for each $i$ where it applies (i.e.\ there
is no change in the case where $\be_{q_i}$ extends $\alvec$ directly by a $\nu$-index), we obtain the image of $i_\al$ after
omitting superfluous chains ending in $\ka$-indices that do not match elements in $P_a$ from the modified $M_{q_i}$. 
The image of $P$ under $\tc\circ i_\al$ is therefore essentially closed above $\alvec$, as desired.
\qed \end{proof}

\begin{theo}\label{closedminitheo} 
Let $\alvec\in\TC\cup\{()\}$ and $M_\alvec$ be closed above $\alvec$ with associated forest $P_a$.
Then the identity is the unique isominimal realization of $P_a$ above $\al$.
\end{theo}
\begin{proof} We argue by induction on the cardinality of $M_\alvec$. Consider $\bevec\in M_\alvec$ such that $\be:=\ov(\bevec)$ is 
the largest immediate $\lo$-successor of $\al$ in $\ov[M_\alvec]$, and let $\tauvec$ be the chain associated with $\bevec$ if
$\alvec\sub\bevec$ and with $\alvec$ otherwise.
We obtain the partitioning 
\[M_\alvec=M_0\cup\{\bevec\}\cup M_\bevec,\]
where $M_0:=\{\devec\in M_\alvec\mid\devec\ktc\bevec\}$, and observe that $\alvec\sub\devec$ for all $\devec\in M_0$, as is seen from
Proposition \ref{gbocharprop} and Remark \ref{intermedextsteprmk}, that 
$M_0$ is closed above $\alvec$ unless $M_0=\emptyset$, that $M_\bevec$ is closed above $\bevec$, and that
\begin{equation}\label{lhreacheq}
\lh(\be)=\max(\ov[M_\bevec]),
\end{equation}
and define $C$ to be the chain of tracking chains of consecutively greatest 
immediate $\lo$-successors from $\be$ to $\lh(\be)$ through $\ov[M_\be]$.  
Let $c_\al$ be an $\al$-covering of $P_a$ and set $\gavec:=\tc(\ga)$ where $\ga:=c_\al(\be)$.
For convenience we define 
\[\vcal:=\tc\circ c_\al\circ\ov.\]
{\bf Case 1:} $\bevec$ extends $\alvec$ by $\becp{n+1,1}=_\NF\eta+\taucp{n+1,1}$. If $\eta>0$, by closedness we either have
$\alvec^\frown(\eta)\in M_\alvec$, or $n>0$, $\eta=\taucp{n,m_n}\in\Ez^{>\taunpr}$, and the extension of $\alvec$ by the
$\nu$-index $\mu_\taucp{n,m_n}$ at $(n,m_n+1)$ is an element of $M_\alvec$. Then $M_0$ is non-empty and the i.h.\ applies
to $\alvec$ and $M_0$.
 
We now show that $c_\al$ is pointwise greater than or equal to the identity.
Without loss of generality we may assume that 
the restriction of $c_\al$ to $M_0$ is the identity, that $\ga=c_\al(\be)$ is $\al$-$\leo$-minimal, and that $\ga\le\be$.
In the case $\ga=\be$ we directly apply the i.h., otherwise we have $\gavec=\alvec^\frown(\eta+\xi)$
for some $\xi\in(0,\taucp{n+1,1})$ and set $\si:=\sumend(\xi)=\sumend(\gacp{n+1,1})$. 
Now straightforward translation from $\gavec$ to $\bevec$ leads to a contradiction with the i.h.\ for $\bevec$ and $M_\bevec$, since 
\[\log((1/\si^\star)\cdot\si)<\log((1/\taunestar)\cdot\taucp{n+1,1}),\] 
where $\si^\star$ is the $n+1$-th unit of $\gavec$ according to Definition 5.1 of \cite{CWc}, as $\ga$ and $\be$ have the same 
$<_2$-predecessors.
\\[2mm]
{\bf Case 2:} Otherwise. Then $\bevec$ either extends $\alvec$ by $\becp{n,m_n+1}$ (where $m_n\ge 1$), which by minimality 
of $\be$ and closedness satisfies $\becp{n,m_n+1}\in\Hz$, and in which case we set $(i,j):=(n,m_n)$, or we have $\alvec\not\sub\bevec$,
so that according to Proposition \ref{gbocharprop} and closedness $(i,j+1):=\bp(\alvec,\bevec)$ exists 
with $\bevec_{\restriction_{i,j+1}}=\bevec$, and with $\alvec=\alvec_{\restriction_{i,j+1}}$ in
the case $\chi^\taucp{i,j}(\taucp{i,j+1})=0$, while $\alvec=\me(\alvec_{\restriction_{i,j+1}})$ if $\chi^\taucp{i,j}(\taucp{i,j+1})=1$. 
We then observe that $\becp{i,j+1}=\alcp{i,j+1}+\rho$ for some $\rho\in\Hz$.
In both cases for $\bevec$ we have $\alvec\not=()$ and $\tau:=\taucp{i,j}\in\Ez^{>\taucppr{i,j}}$. 
If $\cml(\bevec)$ exists we set $(r,s):=\cml(\bevec)$, otherwise we let $(r,s):=(i,j)$.

As in Case 1, if the set $M_0$ is non-empty, we may apply the i.h.\ straightforwardly to see that the identity is the
unique isominimal realization of $M_0$ above $\al$. 
Note that the set $\{\bevec\}\cup M_\bevec$ is closed above $\alvec$, and thus it suffices to show the claim for this set. 
To this end, assume $c_\al$ to be an $\al$-covering of $\{\bevec\}\cup M_\bevec$. 
Without loss of generality we may assume that $\ga=c_\al(\be)$ is $\al$-$\letwo$-minimal and less than or equal to $\lh(\be)$.
\begin{claim}\label{mainlinecovclaim} We may assume that $\gavec$ is of the form 
$\bevec[\nu]$ for some $\nu\le\becp{i,j+1}$.
\end{claim}
\begin{firstclaimproof}
We consider the following two cases.
\\[2mm]
{\bf Case A:} $\becp{i,j+1}=\mu_\tau$ and $\cml(\bevec)$ does not exist. 
Then $M_\bevec$ consists of proper extensions of $\bevec$ only. Moreover, setting $\alvecpr:=\alvec_{\restriction_{i,j}}$ 
the set $\{\bevec\}\cup M_\bevec$ is closed above $\alvecpr$ and consists of proper extensions of $\alvecpr$ only.
We consider the case where $\gavec$ does not extend $\alvecpr$ in one step by a $\nu$-index $\nu\le\mu_\tau$.
Note that while $\alvec\ktc\gavec$, by Lemma \ref{subresplem} we have $\alvecpr\propersub\gavec$ since 
$\alpr:=\ov(\alvecpr)\le\al<\ga\le\lh(\be)$, which entails $c_\al(\lh(\be))\le\lh(\be)$.
The $\al$-$\le_2$-minimality of $\ga$ implies that $\ga$ is even $\alpr$-$\le_2$-minimal.
Let $c_\alpr$ be the appropriate restriction of $c_\al$ to become a $\alpr$-covering of $\{\bevec\}\cup M_\bevec$.

Writing $\gavec=(\gavec_1,\ldots,\gavec_l)$, where $\gavec_r=(\gacp{r,1},\ldots,\gacp{r,k_r})$ for $r=1,\ldots,l$, 
according to Proposition \ref{relletwoprop} and our assumptions we have $k_l=2$, $l^\star\klex(i,j)$, and 
$\si:=\sumend(\gacp{l,1})\in(\taupr,\tau)$.
Thus, $\cml(\gavec_{\restriction_{l,1}})$ does not exist, 
so that the tracking chains of image elements of $c_\al$ greater than $\ga$ are
extensions of $\gavec$, whence by Theorem \ref{verttrantheo} we may assume that $\gavec_{\restriction_{l,1}}$ 
is of the form $\alvecpr^\frown(\si)$.

If $\pistinv(\lasi)<\lat$, straightforward upward base transformation by $\pistinv$ and translation from 
$\gavec$ to $\bevec$ yields a contradiction with the i.h.\ for $\bevec$ and $M_\bevec$. 
Otherwise we have   
\[\taupr<\si\le\taubar\le\maxparam(M_\bevec)=:\rho<\tau\] 
by closedness. Let $\xivec\in M_\bevec$ be $\ktc$-minimal such that $\param_{M_\bevec}(\xivec)=\rho$, so that $M_\xivec$ is closed
above $\xivec$ and only consists of extensions of $\xivec$ as $\cml(\bevec)$ and hence also $\cml(\xivec)$ do not exist. 

If $\si<\rho$, we obtain a contradiction with the i.h.\ 
for $\{\devec\}\cup M_\devec$, which is closed above $\alvecpr$, where $\devec:=\alvecpr^\frown(\rho)$ and $M_\devec$ 
is the translation of $M_\xivec$ to $\devec$, since $\vcalpr[\{\xivec\}\cup M_\xivec]$ is an $\alpr$-covering
of $\{\devec\}\cup M_\devec$ contained in the $\si$-th component.

In the remaining case, where $\si=\taubar=\rho$, by closedness we must
have $\db(M_\bevec)=0$, and the same translation of $M_\xivec$ to $\devec=\alvecpr^\frown(\taubar)$
results in a set $M_\devec$ such that $\{\devec\}\cup M_\devec$ is closed above $\alvecpr$, for which
the appropriate restriction of $c_\alpr$ contradicts the i.h., 
since this covering does not exhaust the maximal branch of $M_\devec$. 
\\[2mm]
{\bf Case B:} Otherwise. If $\cml(\bevec)=(r,s)$ exists, we have $(r,s)\kglex(i,j)$, otherwise we must have
$\becp{i,j+1}<\mu_\tau$ and $(r,s)=(i,j)$. In either case, we then have $(r,s)\klex(i,j)$ if and only if $\becp{i,j+1}=\mu_\tau$, 
and due to closedness we have 
\[\devec:=\bevec_{\restriction_{r,s+1}}[\mu_\taucp{r,s}]\in M_\bevec\quad\mbox{ and }\quad\be<\ov(\devec)=:\de.\]
Setting $\alvecpr:=\alvec_{\restriction_{r,s}}$, $\alpr:=\ov(\alvecpr)$, and $\Mpr:=\{\zevec\in M_\bevec\mid\devec\letc\zevec\}$, 
we observe that $\Mpr$ is closed above $\alvecpr$ and that the restriction $c^\prime$ of $c_\al$ to $\alpr$ and $\Mpr$ is an $\alpr$-covering,
wherefore the i.h.\ applies to reveal that the image of $c^\prime$ is pointwise greater than or equal to the identity. 

Let $\ze$ be the least element in $C$ such that $\be\not\le_2\ze$.
Since $\ga=c_\al(\be)$ is $\al$-$\le_2$-minimal and $\ga\le\lh(\be)$, the assumption $\be<\ga$ implies that $\be\not<_2\ga$ 
and that $\ga\leo\lh(\be)=c_\al(\lh(\be))$. Under this assumption we may modify $c_\al$ to be the identity on 
$\{\be\}\cup\ov[M_\bevec]\cap\ze$, resulting in an $\al$-covering pointwise below $c_\al$.  
As the i.h.\ applies in the case $\ga=\be$, we may therefore assume that $\ga<\be$.
Since $\lh(\be)\le c_\al(\lh(\be))$ as shown above, we have $\ga\lo\lh(\be)$, and thus we may assume that $\ga$ is $\al$-$\le_2$-minimal
such that $\ga\lo\be$, concluding the proof of Claim \ref{mainlinecovclaim}. \qed \end{firstclaimproof}

\begin{claim}\label{extensionpartclaim} We may assume that the image $V:=\vcal[M_\sub]$ of the $\letc$-initial segment 
\[M_\sub:=\{\zevec\in \{\bevec\}\cup M_\bevec\mid\bevec_{\restriction_{r,s+1}}\sub\zevec\}\]
of $\{\bevec\}\cup M_\bevec$ consists of extensions of $\vcal(\bevec_{\restriction_{r,s+1}})$ only.
\end{claim}
\begin{secondclaimproof}
We set $\bevec^\#:=\bevec_{\restriction_{r,s+1}}$ and $\gavec^\#:=\vcal(\bevec^\#)$. Note that  
$\gavec^\#=\bevec^\#=\alvec_{\restriction_{r,s+1}}$ in the case $(r,s)\klex(i,j)$.
Let $V_1,V_2$ be the partitioning of $V$ into extensions of $\gavec^\#$ 
and tracking chains $\zevec$ such that $\gavec^\#\not\sub\zevec$, respectively,
and let $M_1,M_2$ be the corresponding preimages. Let us assume that $V_2\not=\emptyset$.
Due to Lemma \ref{subresplem} we have $V_1\ktc V_2$ and hence also $M_1\ktc M_2$. Note that there does not exist any $\le_2$-connection
from $M_1$ into $M_2$ as there does not exist any such connection from $V_1$ into $V_2$. We consider the decomposition of $\ov[M_2]$
into $\leo$-connectivity components, writing
\[M_2=\bigcup^q_{p=1}\{\xivec_p\}\cup M_{\xivec_p},\]
where $\bevec\ktc\xivec_1\ktc\ldots\ktc\xivec_q$. 
Then the ordinals $\xi_p:=\ov(\xivec_p)$, $p=1,\ldots,q$, are $\alpr$-$\le_2$-minimal, where $\alpr:=\ov(\alvecpr)$ 
and $\alvecpr:=\alvec_{\restriction_{r,s}}$.
Hence, by Proposition \ref{relletwoprop}, each $\xivec_p$ is of a form ${\zevec_p}^\frown(\xicp{p,k_p,1},\ldots,\xicp{p,k_p,l_{p,k_p}})$, 
where $l_{p,k_p}\le 2$ and $\sumend(\xicp{p,k_p,1})<\taucp{r,s}$. 
Clearly, $l_{p,k_p}=1$ for $p=2,\ldots,q$, and we may assume that also $l_{1,k_1}=1$, since the case $l_{1,k_1}=2$ is handled similarly,
as $\ov({\zevec_1}^\frown(\xicp{1,k_1,1}))$ then must be $\alpr$-$\le_2$-minimal as well.
Note that we have $\bevec^\#\sub\zevec_p$ and that each $\cml(\xivec_p)$ would have to satisfy $\cml(\xivec_p)\klex(r,s)$ and therefore 
does not exist for $p=1,\ldots,q$.
Hence, each $M_{\xivec_p}$ is closed above $\xivec_p$ and consists of extensions of $\xivec_p$ only. 

We may thus modify the restriction of $c_\al$ to $\ov[M_\sub]$ on $M_2$ by the appropriate translations of the components 
$\{\xivec_p\}\cup M_{\xivec_p}$ to successively append $\lo$-branches to the greatest common $\lo$-predecessor in $\ov[V_1]$ of 
the ordinals in $\ov[V_2]$, which is possible due to property \ref{baseseqeq} in Remark \ref{intermedextsteprmk}. 
This modification results in a covering that is pointwise less than or equal to $c_\al$, which
concludes the proof of Claim \ref{extensionpartclaim}. \qed \end{secondclaimproof}

\noindent{\bf Case 2.1:} $\becp{i,j+1}<\mu_\tau$. Then we are in the scenario of Case B above. 
\\[2mm]
\noindent{\bf Subcase 2.1.1:} $\cml(\bevec)$ does not exist. Then $M_\sub-\{\bevec\}$ is closed above $\bevec$, 
and the assumption $\gavec=\bevec[\nu]$ for some $\nu<\becp{i,j+1}$ leads to a contradiction with the i.h.\ by
straightforward translation of $V$, which according to Claim \ref{extensionpartclaim} consists of
extensions of $\gavec$ only, from $\ga$ up to $\be$. 
Thus the i.h.\ applies, and we are done. 
\\[2mm]
\noindent{\bf Subcase 2.1.2:} Otherwise. Then we have $\cml(\bevec)=(i,j)$; hence $\bevec^\#=\bevec$, 
$\chit(\rho)=1$ where $\rho:=\sumend(\becp{i,j+1})$, and 
setting $\xivec:=\me(\bevec)\in M_\sub$ we have $\be<_2\ov(\xivec)=:\xi$ and $\xi$ is the immediate predecessor of $\ze$ in $C$,
where $\ze$ is defined as in the above Case B. Note that $\bevec\not\sub\zevec:=\tc(\ze)$ and $\ga<_2 c_\al(\xi)<_1 c_\al(\lh(\be))$
in this situation. According to Claim \ref{extensionpartclaim}, $V$ consists of extensions of $\gavec$ only, containing $\tc(c_\al(\xi))$.
Now define $M_\sub^\prime$ to be the translation of $M_\sub$ from $\bevec$ to \[\bevecpr:=\bevec[\becp{i,j+1}+\rho\cdot\om]\] 
with the additional tracking chain $\tc(\lh_2(\bepr))$, where $\bepr:=\ov(\bevecpr)$. 
Then $M_\sub^\prime-\{\bevecpr\}$ is closed above $\bevecpr$ and contains less elements than $M_\bevec$, since $\{\bevec\}\cup M_\bevec$ 
is closed above $\alvec$. 
If $\gavec=\bevec[\nu]$ for some $\nu<\becp{i,j+1}$, $\nu=_\NF\nupr+\nu_0$, we must have $\chit(\nu_0)=1$,
since $c_\al(\xi)\lo c_\al(\lh(\be))$, where $\tc(c_\al(\lh(\be)))$ is not reacheable by extension of $\gavec$.
Let $V^\prime$ be the translation of $V$ from $\gavec$ to \[\gavecpr:=\bevec[\nupr+\nu_0\cdot\om]\]
with the additional tracking chain $\tc(\lh_2(\gapr))$, where $\gapr:=\ov(\gavecpr)$.
Translating $V^\prime$ from $\gavecpr$ to $\bevecpr$ then gives rise to a $\bepr$-covering of $\ov[M_\sub^\prime-\{\bevecpr\}]$
that contradicts the i.h.
\\[2mm]
\noindent{\bf Case 2.2:} $\becp{i,j+1}=\mu_\tau$. 
\\[2mm]
\noindent{\bf Subcase 2.2.1:} $\cml(\bevec)$ does not exist. Then we are in the situation of the above Case A, 
where $M_\bevec\sub M_\sub$, and the assumption $\gavec=\bevec[\nu]$ for some $\nu<\becp{i,j+1}$ leads to a contradiction with the i.h.\ by
straightforward translation of $V=\vcal[\{\bevec\}\cup M_\bevec]$, which according to Claim \ref{extensionpartclaim} consists of
extensions of $\gavec$ only, from $\ga$ up to $\be$. 
\\[2mm]
\noindent{\bf Subcase 2.2.2:} $\cml(\bevec)=(r,s)$ exists.
Here Claim \ref{extensionpartclaim} applies with $\bevec^\#\subsetneq\bevec$
since $(r,s)\klex(i,j)$, cf.\ the above Case B. Note that we therefore have $\gavec^\#=\bevec^\#=\alvec_{\restriction_{r,s+1}}$, 
and setting $\xivec:=\me(\bevec)=\me(\bevec^\#)$ we have $\be^\#:=\ov(\bevec^\#)<_2\ov(\xivec)=:\xi$, and $\xi$ is the immediate predecessor 
of $\ze$ in $C$, where $\ze$ is now defined to be the least element of $C$ such that $\be^\#\not\le_2\ze$.
Note that $\bevec^\#\not\sub\zevec:=\tc(\ze)$ and 
\[\be^\#<_2 c_\al(\xi)<_1 c_\al(\lh(\be)),\] 
showing that $c_\al(\xi)\in V$ while $\bevec^\#\not\sub \tc(c_\al(\lh(\be)))$. 
Now define 
\[\bevecpl:=\bevec^\#[\alcp{r,s+1}+\rho\cdot\om],\]
where $\rho:=\sumend(\alcp{r,s+1})\in\Hz$, and $\beplus:=\ov(\bevecpl)$.
Let $\bevecpr$ be the image of $\bevec$ under the translation resulting from replacement of the index at $(r,s+1)$
by $\alcp{r,s+1}+\rho\cdot\om$ and note that $\cml(\bevecpr)$ does not exist.
Let $M_\sub^\prime$ be the image of $M_\sub$ under the same translation, with the additional tracking chain $\tc(\lh_2(\beplus))$.
Let $\alvecpr$ be obtained from $\alvec$ in the same way and observe that $M_\sub^\prime$ is closed above $\alvecpr$ 
with $\ktc$-minimal element $\bevecpr$, containing less elements than $\{\bevec\}\cup M_\bevec$.
Assuming that $\gavec=\bevec[\nu]$ for some $\nu<\becp{i,j+1}$ such that $\al<\ga$, let $\gavecpr$ be the image of $\gavec$ under
the same index shift at $(r,s+1)$, and 
let $c_\alpr$, where $\alpr:=\ov(\alvecpr)$, result from $c_\al$ likewise, so that it maps the elements of 
$M_\sub^\prime$ to the corresponding translated image elements of $c_\al$ while fixing $\lh_2(\beplus)$.
Then setting $\xipr:=\ov(\xivecpr)$, where $\xivecpr$ results from translating $\xivec$, we have 
\[\beplus<_2 c_\alpr(\xipr)\lo\lh_2(\beplus),\] 
and $c_\alpr$ is an $\alpr$-covering of $M_\sub^\prime$ contradicting the i.h.
\\[2mm]
\noindent We therefore must have $\nu=\becp{i,j+1}$, whence the claim for $\alvec$ and $M_\alvec$ follows from the i.h.\ for $\bevec$
and $M_\bevec$.
\qed \end{proof}

\noindent\begin{rmk} Note that any covering of an essentially
closed set $M$ extends to a covering of its closure $\bar{M}$ under initial chains. Hence essentially closed sets are uniquely
isominimally realized by the identity.\end{rmk}

Theorems \ref{isomtheo} and \ref{closedminitheo} now readily combine to the following main result on isomorphic copies
of respecting forests of order $2$ in $\Ctwo$ that are unique in being pointwise minimal among all coverings.  
\begin{cor}\label{respforestcor}
Every respecting forest $P$ of order $2$ (and hence every pure pattern of order $2$) has a unique isominimal realization $i[P]$ in $\Ctwo$. 
$i[P]$ is isomorphic to $P$, essentially closed, and hence closed under $\lh$ and $\lh_2$.
\end{cor}

Isominimal realizations are therefore tight within $\Ctwo$ as there do not exist $\leo$- nor $\letwo$-connections to elements of $\Ctwo$
that extend beyond the respective largest connections in the realization.

\begin{cor}[Ordinal notations]\label{notationcor} 
Let $\al<\oneinf$ and $M:=\{\tc(\al)\}^{\operatorname{ecl}}$ be its essential closure.
Then the respecting forest $P$ associated with $M$ together with a marker for the element matching $\al$ provides 
a pattern notation for $\al$. This notation is of least cardinality possible.
\end{cor}
\begin{proof} Let $Q$ be a respecting forest of order $2$, of which the unique isominimal realization $c[Q]$ within $\Ctwo$ contains $\al$.
$c[Q]$ is essentially closed.
Inspection of Definitions \ref{closuredefi} and \ref{esscldefi} shows that we (must) make a choice (choosing a normal form) when performing 
a closure, but in a way that adds as few new elements as possible. Hence $Q$ must have at least as many elements as $P$.
\qed \end{proof}

Together with the obvious, elementary recursive comparison relations, we therefore obtain an elementary recursive notation system
for the ordinal $\oneinf$.
 
\begin{cor}
The union of all isominimal realizations of respecting forests of order $2$ comprises the initial segment $\oneinf$ of the ordinals,
characterizing the core of $\Rtwo$.
\end{cor}
\begin{proof} 
By Theorem 7.4 of \cite{CWc} we know that the arithmetical characterization $\Ctwo=(\oneinf;\le,\leo,\letwo)$ coincides with
the structure $\Rtwo\restriction_{\oneinf}$, where $\leo$ and $\letwo$ are defined as $\Sigma_1$- and $\Sigma_2$-elementary substructurehood,
respectively. $\Core(\Rtwo)$ is by definition the union of all isominimal copies of finite isomorphism types of $\Rtwo$.
Corollary \ref{respforestcor} shows that each respecting forest of order $2$ is a finite isomorphism type of $\Ctwo$, which by Theorem 7.4 
of \cite{CWc} is a finite isomorphism type of $\Rtwo$ with coinciding isominimal realizations, hence $\Core(\Rtwo)=\oneinf$. 
\qed \end{proof}

We finally come to a statement regarding the combinatorial strength of respecting forests of order $2$. Recall the enumeration 
function $\ka$ of the $\leo$-minimal ordinals in $\Ctwo$, cf.\ its extension from \cite{C99} for the segment $\epsn$ to $\oneinf$ in 
Definition 4.4 of \cite{CWc} and Section 4 of \cite{W}, which we reviewed in Subsection \ref{tsccsec}.

\begin{cor} Denote the notation for an ordinal $\ga<\oneinf$ given in Corollary \ref{notationcor} by $P(\ga)$.
Let $\al<\be<\oneinf$. Then there is no covering of $P(\ka_{\om^\be})$ into $P(\ka_{\om^\al})$.
Hence any infinite descending sequence of ordinals below $\oneinf$ produces an infinite bad sequence of respecting forests of order $2$
with respect to coverings.\qed
\end{cor}

Together with Carlson's result that respecting forests of order $2$ are well-quasi-ordered with respect to coverings, 
see \cite{C16}, we obtain
the independence of this wqo-result of the theory $\kplnod$, since as seen above, the well-quasi orderedness would imply 
$\mathrm{TI}(\oneinf)$, i.e.\ transfinite induction up to $\oneinf$, i.e.\ the proof-theoretic ordinal of $\kplnod$ 
(equivalently $\Pi^1_1-\mathrm{CA}_0$).
On the other hand, we have seen by Theorem 7.4 of \cite{CWc} that $\mathrm{TI}(\oneinf)$ suffices to show that every finite substructure
of $\Rtwo$ has a covering contained in $\oneinf$.

\section{Conclusion}
The structure $\Ctwo$, which arithmetically characterizes the structure $\Rtwo$
of pure elementary patterns of resemblance of order $2$ up to $\oneinf$ as proven in \cite{CWc}, 
was shown to be elementary recursive in \cite{W}, which we reviewed in Section \ref{prelimsec}. 
Here we have established mutual elementary recursive order isomorphisms between classical ordinal notations and pattern notations,
showing that pattern notations based on pure $\Sigma_2$-elementarity characterize the proof theoretic ordinal $\oneinf$ of the fragment 
$\Pi^1_1$-$\mathrm{CA}_0$ of second order number theory, or equivalently, the set-theoretic system $\kplnod$, which axiomatizes 
limits of admissible universes (i.e.\ models of $\kpom$, Kripke-Platek set theory with infinity).  

We have seen that finite isomorphism types of $\Ctwo$, hence of $\Rtwo$, comprise (up to isomorphism) the 
class of respecting forests of order $2$, cf.\ \cite{C01} and \cite{C09}. 
We have shown that the union of isominimal realizations of respecting forests of order $2$ is indeed the core of $\Rtwo$ and 
is to equal the proof-theoretic ordinal of $\kplnod$.
As a corollary we have proven that the well-quasi orderedness of respecting forests with respect to coverings, 
which was shown by Carlson in \cite{C16}, implies (in a weak theory) transfinite induction up to the proof-theoretic ordinal $\oneinf$ of $\kplnod$.

We expect, as mentioned in \cite{W}, that the approaches taken here and in our treatment of the structure $\Ronepl$, 
see \cite{W07b} and \cite{W07c}, naturally extend to an analysis of the structure $\Rtwopl$ and possibly to structures 
of patterns of higher order.
A subject of ongoing work is to verify our claim that the core of $\Rtwopl$ matches the proof-theoretic strength of a limit of 
$\mathrm{KPI}$-models, which in turn axiomatize admissible limits of admissible universes.  

\section*{Acknowledgements}
I would like to express my gratitude to Professor Ulf Skoglund for encouragement and support of my research and
thank Dr.\ Steven D.\ Aird for editing the manuscript. 
I would like to acknowledge the Institute for Mathematical Sciences of the National University of Singapore
for its partial support of this work during the ``Interactions'' week of the workshop {\it Sets and Computations} in April 2015.

\printindex

\end{document}